\documentclass[BCOR=12mm, DIV=default,twoside, pagesize, draft]{scrartcl}

\addtokomafont{caption}{\small}
\setkomafont{captionlabel}{\sffamily\bfseries}

\usepackage{tikz} 
  
\usetikzlibrary{arrows,automata}

\usepackage{amsmath, amssymb, amsthm}
\usepackage[utf8x]{inputenc}
\usepackage{libertine}
\usepackage[T1]{fontenc}
\usepackage[english]{babel}
\usepackage{url}  
\usepackage{faktor,csquotes,color}
\usepackage{pgfplots} 
\usepackage{enumitem}

\usepackage{changes}

\definechangesauthor[name={Soeren}, color=red]{soe}
\definechangesauthor[name={Kristoffer}, color=blue]{kri}
\newcommand{\stkout}[1]{\ifmmode\text{\sout{\ensuremath{#1}}}\else\sout{#1}\fi}
\setdeletedmarkup{\stkout{#1}}

\usepackage{units,varioref}

\usepackage[shortcuts]{extdash}
\usepackage{setspace,scrpage2}
\usepackage{mathtools}
\mathtoolsset{showonlyrefs,showmanualtags}
\numberwithin{equation}{section}

\allowdisplaybreaks 
\clearscrheadfoot
\automark[section]{section}
\ohead{\headmark}
\ofoot[\pagemark]{\pagemark}

\newtheoremstyle{break}{\topsep}{\topsep}{\itshape}{}{\bfseries}{.}{\newline}{}
\newtheoremstyle{exampl}{\topsep}{\topsep}{\upshape}{}{\bfseries}{.}{\newline}{}

\theoremstyle{plain}
\newtheorem{thm}{Theorem}[section]
\newtheorem{lem}[thm]{Lemma}

\theoremstyle{break}

\theoremstyle{definition}
\newtheorem{defi}[thm]{Definition}

\newtheorem{ex}[thm]{Example}

\theoremstyle{exampl}

\theoremstyle{remark}
\newtheorem{rem}[thm]{Remark}

\def\R{\mathbb{R}}

\def\N{\mathbb{N}}

\def\state{E}

\DeclareMathOperator{\E}{\mathbb E}

\definecolor{mygray}{gray}{.5}

\subject{\begin{flushleft}
{\normalfont 
{\footnotesize First published in:} \textcolor{blue}{\normalsize SIAM Journal on Control and Optimization, 56(6), 2018.}\\*
\vspace{-2mm}{\normalsize DOI:}\textcolor{blue}{\normalsize 10.1137/17M1153029} 
\vspace{-2mm}{\footnotesize Copyright ©  SIAM. Unauthorized reproduction prohibited.}
}\end{flushleft} 
}

\title{On Finding Equilibrium\ Stopping Times for Time-Inconsistent Markovian Problems}
\author{S\"oren Christensen\footnote{Department of Mathematics, SPST, University of Hamburg, Bundesstr. 55, D-20146 Hamburg, Germany. E-mail address: soeren.christensen@uni-hamburg.de} \and Kristoffer Lindensj\"o\footnote{Department of Mathematics, Stockholm University, SE-106 91 Stockholm, Sweden. E-mail address: kristoffer.lindensjo@math.su.se.}}

\date{November, 2018}
\vspace{-4mm}

\begin{document}
\maketitle

\begin{abstract}
Standard Markovian optimal stopping problems are consistent in the sense that the first entrance time into the stopping set is optimal for each initial state of the process. 
Clearly, the usual concept of optimality cannot in a straightforward way be applied to non-standard stopping problems without this time-consistent structure. This paper is devoted to the solution of time-inconsistent stopping problems with the reward depending on the initial state using an adaptation of Strotz's consistent planning. More precisely, we give a precise equilibrium definition --- of the type subgame perfect Nash equilibrium based on pure Markov strategies. 
In general, such equilibria do not always exist and if they exist they are in general not unique. We, however, develop an iterative approach to finding equilibrium stopping times for a general class of problems and apply this approach to one-sided stopping problems on the real line.  We furthermore prove a verification theorem based on a set of variational inequalities which also allows us to find equilibria. In the case of a standard optimal stopping problem, we investigate the connection between the notion of an optimal and an equilibrium stopping time. As an application of the developed theory we study a selling strategy problem under exponential utility and endogenous habit formation.
\end{abstract}

\vspace{1mm}
 

\section{Introduction} \label{introKri}
Consider a Markovian process $X$ with state space $E \subseteq \mathbb{R}^d$ and the problem of choosing a stopping time $\tau$ in order to maximize the expected discounted reward 
\begin{align} \label{intro-eq}
\E_x(e^{-r\tau}F(X_\tau,x)), \enskip \textrm{for each current state $x \in E$}.
\end{align}
Note that the dependence of the reward $F(X_\tau,x)$ on the current state $x$ implies that this is not a standard optimal stopping problem. Specifically, the problem is inconsistent in the sense that we cannot generally expect the existence of an optimal stopping time that is independent of the current state $x$. In other words, if an optimal stopping time is optimal for the current state $x$, then it will generally not be optimal at a later time, call it $t$,  after adjusting the reward, since the then current state $X_t$ will typically be different from $x$. Optimal stopping and more general optimal control problems with this property are  called \emph{time-inconsistent}.

It is clear that the time-inconsistency implies that the usual notion of optimality cannot be applied straightforwardly --- it must first be clarified how the time-inconsistent stopping problem should be interpreted. One way of dealing with this issue is of course to treat the problem as a parametrized, by the current state $x$, optimal stopping problem and ignore the issue that  
the corresponding optimal 
stopping time will not generally be optimal at later times. In the literature this is known as a \emph{pre-commitment} strategy. In the present paper we instead interpret the time-inconsistent stopping problem using a game-theoretic approach where we let each state $x$ correspond to an agent, who all  play a sequential game against each other regarding when to stop the process $X$ --- and then we look for equilibrium strategies i.e. \emph{equilibrium stopping times}. The type of equilibria we consider are subgame perfect Nash equilibria based on pure Markov strategies (known as pure Markov perfect equilibria). Subgame perfect Nash equilibrium is a refinement of the notion of Nash equilibrium for dynamic games suggested by Selten, see e.g. \cite{selten1965spieltheoretische,selten1975reexamination}, relying on the concept of \emph{consistent planning} of Strotz \cite{strotz}.  We remark that although Strotz \cite{strotz} studied dynamic utility maximization (a control problem) under inconsistency essentially due to the reward being dependent on the current time, the essential idea of the present paper relies on the creation of Strotz, although our presentation relies on inconsistency due to the space variable.   

In game-theory strategies are either pure or mixed. In the present paper we consider only pure stopping strategies, which we define as entry times into sets in the state space, see Remark \ref{motiv-eq-def} for a motivation. In \cite{christensen2018time} we consider mixed strategies for time-inconsistent stopping problems of a different class compared to the present paper. Remark \ref{motiv-eq-def} contains an explanation of some of the game-theory terms used in the present paper.

The structure of the paper is as follows. In Section \ref{sec:prev-lit} we describe previous literature related to time-inconsistent problems. In Section \ref{sec:problem-formulation} we formulate the general time-inconsistent stopping problem introduced above in more detail and define the notions of pure Markov strategies and subgame perfect Nash equilibria in this setting. Here we also show that optimal stopping times for standard (time-consistent) stopping problems are equilibrium stopping times, and that the reverse holds for one-dimensional absorbed Wiener processes (a similar result for diffusions is presented in Section \ref{sec:time-incon-var-eq}). We also present an example with two essentially different equilibria and an example which proves that an equilibrium (of the type we consider) does not always exist. In Section \ref{sec:forward_iteration} we develop an iterative approach to finding equilibrium stopping times in a general setting under certain assumptions. As an application of this iterative approach we, in Section \ref{sec:one-sided}, study a class of one-sided problems on the real line. In Section \ref{sec:time-incon-var-eq} we present a verification theorem for time-inconsistent optimal stopping based on a set of variational inequalities that we call the \emph{time-inconsistent variational inequalities}. Illustrative examples are studied in Sections \ref{sec:forward_iteration}, \ref{sec:one-sided}, and \ref{sec:time-incon-var-eq}. In Example \ref{eq-selling-strat} we apply the verification theorem to find equilibrium selling strategies for an investor with  exponential utility and endogenous habit formation.

\subsection{Previous literature}\label{sec:prev-lit}
Time-inconsistency in financial economics typically arises for either of the following reasons: 
\begin{enumerate}[label=(\roman*)]
\item Endogenous habit formation, \label{endog-hab}
\item Non-exponential discounting, \label{non-exp}
\item Mean-variance utility. \label{mean-var}
\end{enumerate}
{Stopping problems with \ref{endog-hab} and \ref{non-exp} can be formulated and studied in the framework of the present paper whereas stopping problems of type \ref{mean-var} can be dealt with in the framework studied in \cite{christensen2018time}. Stopping problems with \ref{non-exp}---\ref{mean-var} are described below. A stopping problem with \ref{endog-hab} is studied in Example \ref{eq-selling-strat}. See also \cite{tomas-disc} for a short description of \ref{endog-hab}---\ref{mean-var}.}

There is a substantial financial economics literature that studies specific time-inconsistent problems; in either continous or discrete time, for either stochastic or deterministic models, and using either game-theoretic or pre-commitment approaches. Historically important papers include \cite{goldman1980consistent,peleg1973existence,pollak1968consistent,strotz}. We remark that most of these papers consider problems of control type. Papers in financial economics studying time-inconsistent stopping problems include \cite{barberis2012model,ebert2015until,ebert2017discounting,grenadier2007investment}.

As mentioned above, the first paper to use a game-theoretic approach --- essentially based on subgame perfect Nash equilibria --- to time-inconsistency, there termed \emph{consistent planning}, was \cite{strotz}, where a deterministic problem under non-exponential discounting in discrete time is studied. Further financial economics research in this direction can be found in \cite{barro1999ramsey,goldman1980consistent,krusell2003consumption,peleg1973existence,pollak1968consistent}.

Early papers of a more mathematical kind to consider the game-theoretic approach --- based on subgame perfect Nash equilibria --- in time-inconsistent problems in continuous time are \cite{ekeland-arxiv,Ekeland}, who study optimal consumption and investment under non-exponential (hyperbolic) discounting. Inspired by the approach of e.g. \cite{ekeland-arxiv,Ekeland}, the first papers to develop a general mathematical theory for finding subgame perfect Nash equilibria for time-inconsistent stochastic control problems in Markovian models are \cite{tomas-continFORTH,tomas-disc}. The main feature of that theory is a generalization of the standard HJB equation called the extended HJB system and the main result is a verification theorem saying that if a solution  to the extended HJB system exists then it corresponds to an equilibrium. In \cite{lindensjo2017timeinconHJB} it is shown that a \emph{regular} equilibrium is necessarily a solution to an extended HJB system. Other papers studying specific time-inconsistent control problems from a more mathematical perspective include \cite{tomas-mean,Czichowsky,ekeland2012time,hu2012time,li2002dynamic}.

Papers of a more mathematical kind to study time-inconsistent stopping include \cite{huang2017time} who study a stopping problem with non-exponential discounting and \cite{miller2017nonlinear,pedersen2016optimal} (see the discussion below).  In \cite{bayraktar2018} a game-theoretic approach inspired by Strotz's consistent planning is used to study a time-inconsistent stopping problem with a mean standard deviation criterion in discrete time. 
In \cite{christensen2013american} a class of stopping problems -- which can be seen as American options with guarantee -- with the reward depending on the initial state are studied using a pre-commitment approach. 
We refer to \cite{tomas-disc,huang2017time,miller2017nonlinear,pedersen2016optimal} for short surveys of the literature on time-inconsistent problems. 

Endogenous habit formation problems (see Example \ref{eq-selling-strat}) are time-inconsistent because the reward depends on the current state. Stopping problems of this kind can therefore be studied in the framework of the present paper. A version of the non-exponential discounting stopping problem corresponds to maximizing
	\begin{equation} \label{noe-exp-eq}
	\E_{t,x}(\delta(\tau-t)\tilde F(X_\tau))
	\end{equation}
with respect to stopping times $\tau$, where the discounting function $\delta:[0,\infty)\rightarrow [0,1]$ is a decreasing (non-exponential) function satisfying $\delta (0)=1$. Problem \eqref{noe-exp-eq} can in our framework be obtained by letting one of the dimensions of $X$ correspond to time, i.e., by considering the time-space process.

Mean-variance problems are, however, time-inconsistent for the fundamentally different reason that the expression to be maximized is a non-linear function of the expected value of a reward. Hence, mean-variance problems cannot be studied in the present framework 
(a mean-variance problem is however studied in \cite{christensen2018time}). A version of the mean-variance stopping problem is to find a stopping time $\tau$ that maximizes
\begin{equation} \label{mean-var-eq}
\E_{x}(X_\tau)- cVar_{x}(X_\tau),  \enskip \mbox{ where $c>0$ is a fixed constant.} 
\end{equation}
In \cite{pedersen2016optimal}, this mean-variance stopping problem is studied for an underlying geometric Brownian motion (i.e. a mean-variance selling problem in a Black-Scholes market). The problem is interpreted and solved in two different ways, by the introduction of two different definitions of optimality. 
\emph{Static optimality}, corresponds to finding, for a fixed $x>0$, a stopping time that maximizes \eqref{mean-var-eq}. The static optimality definition corresponds to a pre-commitment approach. 
\emph{Dynamic optimality},  corresponds to finding a stopping time $\tau^*$ such that there is no other stopping time $\sigma$ with $\mathbb{P}_x(\E_{X_{\tau^*}}(X_\sigma)- cVar_{X_{\tau^*}}(X_\sigma)>X_{\tau^*})>0$ for some $x>0$. This is a novel interpretation of time-inconsistent problems, that does not rely on game-theoretic arguments, see, however, Remark \ref{rem:} below. 
We remark that the concept of dynamic optimality is applicable also to the time-inconsistent stopping problem considered in the present paper, as well as to time-inconsistent stochastic control problems (both of the type considered in the present paper, cf. \ref{endog-hab} and \ref{non-exp}, and of the non-linear type, cf. \ref{mean-var}), see \cite{pedersen2018constrained,pedersen2017optimal}. These references contain the first known time-consistent strategies that are optimal for constrained mean-variance portfolio selection problems in continuous time --- we also remark that there are no known subgame perfect Nash equilibrium strategies for these constrained problems, although \cite{tomas-mean} studies an unconstrained version.
A crucial difference between the game-theoretic approach based on Strotz's consistent planning and the dynamic optimality approach is that the game-theoretic equilibrium solution can be interpreted as the best control among those that will actually be used in the future, while the dynamic optimality solution can be interpreted as being the best with respect to all present states.
We refer to \cite{pedersen2016optimal} for a  further discussion of the difference between these two different approaches, see in particular the paragraph before \cite[Example 9]{pedersen2016optimal}.
We remark that \cite{pedersen2016optimal} contains also a subgame perfect Nash equilibrium approach (based on Strotz's idea) for stopping problems, see mainly \cite[Example 9]{pedersen2016optimal}; this point is elaborated in the paragraph before Example \ref{ex-multiple-eq} below. Time-inconsistent stopping problems with more general non-linear functions of the expected reward  are studied in \cite{miller2017nonlinear} using an approach which is inspired by \cite{pedersen2016optimal}.

\section{Problem formulation} \label{sec:problem-formulation}
On the filtered probability space $(\Omega,\mathcal{F},(\mathcal{F}_t)_{t\geq 0},\mathbb{P}_x)$ we consider a strong Markov process $X=(X_t)_{t\geq 0}$ taking values in $(E,\mathcal{B})$ where $E\subseteq \R^d$ and $\mathcal{B}$ is the corresponding Borel $\sigma$-algebra and $X_0=x\in E$. We assume that the filtration satisfies the usual conditions and $X$ to have c\'adl\'ag sample paths and to be quasi left continuous and that $x \mapsto  \mathbb{P}_x(F)$ is measurable for each $F \in \mathcal{F}$. The associated expectations are denoted by $\mathbb{E}_x$. Without loss of generality we assume that $(\Omega,\mathcal{F})$ equals the canonical space so that the shift operator $\theta$ given by $\theta_t(\omega)(s)= \omega(t+s)$ for $\omega = (\omega(t))_{t\geq0} \in \Omega$ and $t,s\geq0$ is well-defined. The class of stopping times with respect to $(\mathcal{F}_t)_{t\geq 0}$ is denoted by $\mathcal{M}$.

Consider a function $F: E\times E \rightarrow \mathbb{R}$  
and the problem of finding a stopping time $\tau$ such that it maximizes, over  the class of stopping times $\mathcal{M}$,
\[J_{\tau}(x):=\E_x\left(e^{-r\tau}F(X_\tau,x)\textbf{1}_{\{\tau < \infty\}}\right), \enskip \textrm{for each } x\in E,\]
where $r\geq 0$ is a constant and --- to guarantee that all expectations are well-defined --- the function $F(\cdot,y)$ is measurable and bounded from below for each fixed $y \in E$. For the ease of exposition we will in the rest of the paper not explicitly write out indicator functions of the type $\textbf{1}_{\{\tau < \infty\}}$ in expected values, but instead implicitly assume that they are there.

The difference in our formulation to usual Markovian optimal stopping problems is that the reward {$F(X_\tau,x)$} explicitly depends on the initial state $X_0=x$. In the standard formulation, the reward {$F(X_\tau,x)$} is independent of $x$. In that classical case, it is well-known that --- under minimal assumptions --- an optimal stopping time is Markovian in the sense that it is a first entrance time into the stopping set, see, e.g., \cite[I.2.2]{peskir2006optimal}. In particular, this solution is consistent meaning that one rule is optimal for each initial state, {i.e. such problems are consistent with  Wald-Bellman's principle of optimality.}

This kind of consistency can of course not be expected in our formulation. We therefore have to be careful how to {reinterpret} the concept of optimality. Clearly, we could choose different stopping times for different starting points $x$. This, however, does not represent the following interpretation of our problem:
 
 We interpret the time-inconsistent stopping problem above as a stopping problem for a person whose preferences, identified with the reward function $F(\cdot,x)$, change as the state $x$ changes. Based on this we think of the person as comprising versions of herself, one version for each state $x$. These versions of the person can then be thought of as agents who play a {sequential} game against each other, where the game regards when to stop the process $X$. {Note that the number of players in this game is generally uncountable.} Each agent, i.e. each $x$-version of the person, then has the possibility, at $x$, to either stop, or not stop. A reasonable definition of an equilibrium strategy, in this case an \emph{equilibrium stopping time} $\hat\tau$, should therefore be such that the following holds:

Under the assumption that each other version of the person uses $\hat\tau$ then, 

\begin{enumerate}[label=(\roman*)]
\item no $x$-version of the person wants to stop  in her state before $\hat\tau$, and 

\item no $x$-version of the person wants to  continue for an "infinitesimal" time if $\hat\tau$ calls for stopping.
\end{enumerate}

We thus define an equilibrium stopping time $\hat\tau$ using conditions which guarantee that no agent {wants} to deviate from $\hat\tau$.  Furthermore, we demand that the decision whether to stop or not should depend directly only on the preferences of each agent $x$ and not, for example, on the outcome of some randomization procedure, or on events from the past  --- that is, we consider pure stopping strategies cf. Definition \ref{def:pure_strat} and Remark \ref{motiv-eq-def} (while mixed stopping strategies are considered in \cite{christensen2018time}). These conditions are, in reverse order, formalized in the following definitions.

\begin{defi} \label{def:pure_strat} A stopping time $\tau \in \mathcal{M}$ is said to be a pure Markov strategy stopping time if it is the entrance time of the state process into a set in the state space, more specifically, if $\tau=\inf\{t\geq 0: X_t \in S\}$ for some measurable $S \subseteq E$. Denote the set of such stopping times by $\mathcal{N}$.
\end{defi}

\begin{defi}\label{def:equ_stop_time} A stopping time $\hat\tau\in\mathcal{N}$ is said to be a (pure Markov strategy) equilibrium stopping time if, for all $x\in \state$,
\begin{align}
J_{\hat\tau}(x)-F(x,x)&\geq 0, \enskip \textrm{ and}\label{eqdef1}\\
\liminf_{h\searrow 0}\frac{J_{\hat\tau}(x)-J_{\hat\tau\circ \theta_{\tau_h}+\tau_h}(x)}{\E_x(\tau_h)}&\geq 0,\label{eqdef2}
\end{align}
where $\tau_h=\inf\{t\geq 0:|X_t-X_0|\geq h\}$.
\end{defi}
\begin{defi} \label{def:equ_func} If ${\hat\tau}$ is a (pure Markov strategy) equilibrium stopping time then the function $J_{\hat\tau}(x),\, x \in E$, is said to be the (pure Markov strategy) equilibrium value function corresponding to ${\hat\tau}$. The function 
\[f_{\hat\tau}(x,y): = \E_{x}(e^{-r\hat\tau}F(X_{\hat\tau},y)), \enskip (x,y) \in E \times E\]
is said to be the auxiliary function corresponding to ${\hat\tau}$.
\end{defi}
It follows that the equilibrium value function satisfies
\[J_{\hat\tau}(x) = f_{\hat\tau}(x,x) = \E_{x}(e^{-r\hat\tau}F(X_{\hat\tau},x)), \enskip x \in E.\]

This paper is devoted to the question of how to find equilibrium stopping times ({of the type in} Definition \ref{def:equ_stop_time}). 

The interpretation of \eqref{eqdef1} is that each $x$-agent should prefer the equilibrium strategy over stopping directly. The interpretation of \eqref{eqdef2} is that each $x$-agent should prefer the equilibrium strategy over not stopping on the short (stochastic) time interval $[0,\tau_h)$, over which we interpret the $x$-agent as being in charge, given that the equilibrium strategy is played from $\tau_h$ and onwards --- here we remark that the numerator in \eqref{eqdef2} can in principle be negative for each $h>0$ and still comply with condition \eqref{eqdef2} by vanishing with order $\E_x(\tau_h)$.

\begin{rem} \label{rem-defNEW}
The equilibrium definition (Definition \ref{def:equ_stop_time}) is essentially an adaptation of Strotz's consistent planning (subgame perfect Nash equilibrium) approach to the type of stopping problems studied in the present paper. The definition is also inspired by similar definitions for time-inconsistent stopping problems in financial economics, see \cite{ebert2017discounting}. The definition can also be seen as an adaptation of the subgame perfect Nash equilibrium for time-inconsistent stochastic control problems (which is itself an adaptation of Strotz's consistent planning), see e.g. \cite{tomas-continFORTH,tomas-disc} and the references therein, when identifying stopping times with binary controls (this is also noted in \cite[Section 3.2]{ebert2017discounting}). We remark that a similar identification is used in \cite[Example 9]{pedersen2016optimal} where a subgame perfect Nash equilibrium approach to stopping problems is also studied (see the paragraph before Example \ref{ex-multiple-eq} below for further details). 
\end{rem}

\begin{rem} \label{motiv-eq-def} Let us informally describe some of the game theoretic jargon used above, for a reference see e.g. \cite{maskin2001markov}. A Markov strategy depends on past events that are payoff-relevant. Markov strategies can be pure or mixed. A pure strategy is one that determines the actions of the agents without randomization. In our setting, the actions of the agents are to stop or not to stop, hence first entrance times correspond to pure strategies. A mixed strategy is one that randomly selects pure strategies. In our situation, this could be realized by extending the underlying filtration in a suitable way and consider general stopping times with respect to this filtration.
See Example \ref{mixedSec} below for an illustration. A subgame perfect Nash equilibrium is a strategy that forms a Nash equilibrium at any time $t$, and a Markov perfect equilibrium is a subgame perfect Nash equilibrium in which all players use Markov strategies. Thus, Definition \ref{def:equ_stop_time} corresponds to a subgame perfect Nash equilibrium, and more specifically a pure Markov perfect equilibrium.
\end{rem}

If the reward function $F(x,y)$ does not depend on $y$, then our time-inconsistent stopping problem is a standard (time-consistent) stopping problem corresponding to
\begin{align} \label{newresA0}
\E_x(e^{-r\tau}F(X_\tau)).
\end{align}
It is now natural to ask: \emph{is the equilibrium value function of the standard stopping problem corresponding to \eqref{newresA0} uniquely given by the optimal value function for \eqref{newresA0}?}
 
We will answer this question as follows: Theorem \ref{newresA} shows that if an optimal stopping time for the standard stopping problem \eqref{newresA0} exists, then the corresponding optimal value function is also an equilibrium value function for \eqref{newresA0}. Theorem \ref{newresA2}  and Theorem \ref{newresB} show that the reverse holds for some cases.
\begin{thm} \label{newresA} 
An optimal stopping time for the standard stopping problem \eqref{newresA0} is an equilibrium stopping time for \eqref{newresA0}.
\end{thm} 
\begin{proof}
If $\tau$ is an optimal stopping time in \eqref{newresA0} then, trivially, the corresponding optimal value function satisfies $J_{\tau}(x)  \geq F(x)$, which means that equilibrium condition \eqref{eqdef1} is satisfied.  To see that also equilibrium condition \eqref{eqdef2} is satisfied note that, trivially, $J_{\tau}(x) \geq J_{\tau\circ \theta_{\tau_h}+\tau_h}(x)$, which means that the numerator in \eqref{eqdef2} is non-negative, for each $h$, and hence that condition \eqref{eqdef2} holds. It follows that $\tau$ is an equilibrium stopping time, by Definition \ref{def:equ_stop_time}.
\end{proof}

\begin{thm} \label{newresA2} 
Suppose that $E=[0,1]$ and $X$ is a Wiener process absorbed at $0$ and $1$, and $r=0$. Suppose that an equilibrium stopping time $\hat\tau$ for the standard stopping problem \eqref{newresA0} exists and that the equilibrium value function $J_{\hat \tau}(x)=\E_x(F(X_{{\hat \tau}}))$ is continuous. Then, $\hat\tau$ is also an optimal stopping time for \eqref{newresA0}.
\end{thm} 
\begin{proof} 
By definition the equilibrium value function is given by $J_{\hat \tau}(x)=\E_x(F(X_{{\hat \tau}}))$ where ${\hat \tau}$ is the entry time into some set in the state space. By condition \eqref{eqdef1} it holds that $J_{\hat \tau}(x)$ dominates the reward function $F(x)$. Under the stated assumptions, superharmonicity is equivalent to concavity. Hence, if we can prove that $J_{\hat \tau}(x)$ is a concave function then it follows that it is also a minimal dominating superharmonic function and, by the standard theory, that ${\hat \tau}$ is an optimal stopping time.

Use the strong Markov property and basic properties of the Wiener process to find that condition \eqref{eqdef2} can for $x \in (0,1)$ be written as 
\begin{align}
\liminf_{h\searrow 0}\frac{J_{\hat\tau}(x)- \frac{1}{2}J_{\hat\tau}(x+h)-\frac{1}{2}J_{\hat\tau}(x-h)}{h^2}&\geq 0. 
\end{align}
This implies that
\begin{align}
\limsup_{h\searrow 0}\frac{-J_{\hat\tau}(x+h)-J_{\hat\tau}(x-h) +2J_{\hat\tau}(x) }{h^2}&\geq 0. \label{newresaug18-1}
\end{align}
By a result from real analysis, see e.g. \cite[Lemma 4.17]{babcock1975properties} or \cite{zygmund2002trigonometric}, it follows, from \eqref{newresaug18-1}, that the function $-J_{\hat\tau}(x)$ is convex (on $(0,1)$), and hence $J_{\hat\tau}(x)$ is concave.

\end{proof}

In Example \ref{ex-multiple-eq} we present an example with multiple equilibria and in Example \ref{mixedSec} we present an example with no equilibrium. Another example of a stopping problem with multiple equilibria is presented in \cite[Example 9]{pedersen2016optimal}. There, however, a different interpretation of the notion of a subgame perfect Nash equilibrium for stopping problems, compared to the present paper, 
	is used as our condition \eqref{eqdef2} does not become relevant
 and the players are identified with the time-coordinates of the time-space process.

\begin{ex} \label{ex-multiple-eq}  Let $E=[0,1]$ and $X$ be a Wiener process absorbed at $0$ and $1$. Consider the reward
\begin{equation}
F(x,y)=\begin{cases}
	1, &x \in \{0,1\},\\
	-|x-y| ,& x\in (0,1).
	\end{cases}
	\end{equation}
It is easy to verify that $\hat\tau=\inf\{t\geq 0: X_t \in \{0,1\}\}$ is an equilibrium stopping time: to see this note that, $J_{\hat\tau}(x)= 1 \geq F(x,x)$, which means that \eqref{eqdef1} holds, and $J_{\hat\tau}(x) = 1 \geq J_{\tau}(x)$ for any stopping time $\tau$, which implies that the numerator of \eqref{eqdef2} is non-negative for each $h$, which implies that \eqref{eqdef2} holds. It is also easy to see that $\tilde\tau=0$ is an equilibrium stopping time: Condition \eqref{eqdef1} holds trivially. Moreover, {for $x\in (0,1)$,} it holds
		that $J_{\tilde\tau}(x) = 0 \geq J_{\tilde\tau\circ \theta_{\tau_h}+\tau_h}(x)$ for sufficiently small $h$ (i.e. immediate stopping is better than continuing a short while), which means that \eqref{eqdef2} holds. For $x\in \{0,1\}$, \eqref{eqdef2} is easily verified. The corresponding equilibrium value functions are given by
$J_{\hat\tau}(x) = 1 $ for $x \in E$, and 
$J_{\tilde\tau}(x) = 0 $ for $x\in (0,1)$, $J_{\tilde\tau}(x) = 1$ for $ x\in \{0,1\}$.
\end{ex}

\begin{ex}
 \label{mixedSec}
As mentioned above, in standard Markovian optimal stopping problems, 	
we only have to consider
  first entrance times and the filtration generated by $X$.
Moreover, any additional information included in some larger filtration cannot improve the optimal value function 
	as long as the process is Markovian also with respect to the larger filtration.
Similarly, in Markovian Dynkin-type stopping games it is also the case that equilibria can be found (under technical assumptions) as first-entrance times, see \cite{ekstrom2008optimal}. This is however not the case for equilibrium stopping problems in general as we will see in the following example.

Consider a discrete time process $X$ that lives on the state space $E=\{\partial_1, a,b,\partial_2\}\subseteq \R$ where $\partial_1$ and $\partial_2$ are absorbing states and \[\mathbb{P}_{a}(X_1=\partial_1)=\mathbb{P}_{a}(X_1=b) =\mathbb{P}_{b}(X_1=a)=\mathbb{P}_{b}(X_1=\partial_2)=\frac{1}{2}.\]
Let $r=0$ and define $X_\infty= \lim_{t\rightarrow \infty}X_t$.

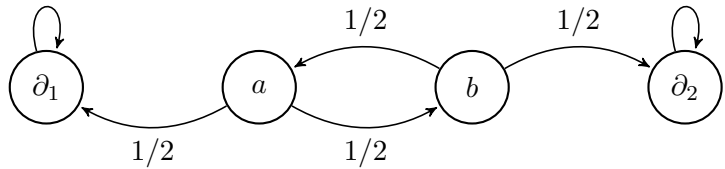
\begin{center}
	\begin{tikzpicture}[->,>=stealth',shorten >=1pt,auto,node distance=2.8cm,semithick]
	\tikzstyle{every state}=[fill=white,draw=black,thick,text=black,scale=1]
	\node[state]         (A){$\partial_1$};
	\node[state]         (B) [right of=A] {$a$};
	\node[state]         (C) [right of=B] {$b$};
	\node[state]         (D) [right of=C] {$\partial_2$};
	\path (B) edge  [bend left] node[below] {$1/2$} (A);
	\path (D) edge  [loop above]  (D);
	\path (A) edge  [loop above]  (A);
	\path (B) edge  [bend right] node[below] {$1/2$} (C);
	\path (C) edge  [bend right] node[above] {$1/2$} (B);
	\path (C) edge  [bend left] node[above] {$1/2$} (D);
	\end{tikzpicture}
\end{center}

 ($X$ can of course be embedded into a continuous time Markov chain, so that we do not leave the setting of this paper). 
 Let
\[F(x,a)=\begin{cases}
0,&x= \partial_1\\
1,&x= a\\
3,&x= b\\
0,&x= \partial_2,
\end{cases}\]
\[F(x,b)=\begin{cases}
4,&x= \partial_1\\
0,&x= a\\
1,&x= b\\
0,&x= \partial_2
\end{cases}\]
and {$F(\cdot,\partial_i)=0$}  
for $i=1,2$. We will now show that no entrance time of the state process $X$ into a subset $S \subseteq E$ can be an equilibrium stopping time, i.e. no pure {Markov} strategy equilibrium stopping time exists. We do this by investigating all such stopping sets $S$. Since $\partial_1$ and $\partial_2$ are absorbing we can without loss of generality assume that $\partial_1,\partial_2 \in S$. It remains to consider the following four sets:
 
\begin{enumerate}[label=(\roman*)]
\item $S=\{\partial_1,\partial_2, a,b\}$: This stopping set corresponds to the rule that both agent $a$ and agent $b$ should always stop when they get the chance. But this rule cannot correspond to an equilibrium stopping time, since agent $a$ would obtain $1$ when stopping but she obtains $\frac{1}{2}\cdot0+\frac{1}{2}\cdot3=\frac{3}{2}>1$ (in expectation) if she deviates from the rule by never stopping.

\item $S=\{\partial_1,\partial_2, a\}$:  This stopping set corresponds to the rule that $a$ should stop and $b$ should continue. But this cannot correspond to an equilibrium stopping time, since agent $b$ obtains $\frac{1}{2}\cdot0+\frac{1}{2}\cdot0=0$ when continuing and $1>0$ when stopping.

\item $S=\{\partial_1,\partial_2, b\}$: This stopping set corresponds to the rule that $b$ should stop and $a$ should continue. Let $V_{0,0}(b)$ denote the value that agent $b$ obtains when not following this rule. Then $V_{0,0}(b)=\frac{1}{2}\cdot0 + \frac{1}{2}(\frac{1}{2}\cdot4+\frac{1}{2}V_{0,0}(b))\Rightarrow V_{0,0}(b) = \frac{4}{3}$. Note that agent $b$ obtains $1<V_{0,0}(b)$ when stopping. This means that the set $\{\partial_1,\partial_2, b\}$ cannot be the stopping set of an equilibrium stopping time.

\item\label{item:ex_(iv)} $S=\{\partial_1,\partial_2\}$: This stopping set corresponds to the rule that both $a$ and $b$ should continue. Since agent $a$ obtains zero in the absorbing states she prefers to stop since this gives her $1$. 
\end{enumerate}
The above implies that there is no equilibrium stopping time in the set of pure Markov strategy stopping times. However, a  mixed strategy equilibrium stopping time (in the sense defined below for this example) does exist, as we shall now see. Consider the stopping time $\tau_{p,q}$ defined as follows: for any  $t\in \N_0$ given $\{\tau_{p,q}\geq t\}$, if $X_t\in\{\partial_1,\partial_2\}$ then $\tau_{p,q}=t$, if  $X_t=a$ then $\tau_{p,q}=t$ with probability $p$, and if $X_t=b$ then $\tau_{p,q}=t$ with probability $q$ (assume that the filtration $(\mathcal{F}_t)_{t\geq 0}$ is large enough for $\tau_{p,q}$ be be a stopping time with respect to $(\mathcal{F}_t)_{t\geq 0}$). Heuristically, the stopping time $\tau_{p,q}$ corresponds to the agents $a$ and $b$ flipping biased coins in order to decide whether to stop or not. Let $V_{p,q}(x), x\in\{a,b\},$ denote the (expected) value that agent $x$ obtains when $\tau_{p,q}$ is used.

The following ad hoc definition, which is inspired by \cite{touzi2002continuous}, will be used only in the present example:\\
	A stopping time of the type $\tau_{p,q}$ (defined above) is said to be a mixed strategy stopping time. 
A mixed strategy stopping time $\tau_{p',q'}$ is said to be a mixed strategy equilibrium stopping time  if $V_{p,q'}(a)\leq V_{p',q'}(a)$ for all $p\in[0,1]$ and $V_{p',q}(b)\leq V_{p',q'}(b)$ for all $q\in[0,1]$.

Heuristically, a mixed strategy equilibrium stopping time $\tau_{p',q'}$ is a strategy from which neither agent $a$ (nor $b$) wants to deviate from by choosing another mixed strategy $\tau_{p,q'}$ ($\tau_{p',q}$), i.e. they do not want to deviate by choosing another biased coin (including degenerate biased coins, i.e. with $p,q\in\{0,1\}$).

We obtain
\begin{align*}
V_{p,q}(a) & = p\cdot 1 + 
(1-p)\left[\frac{1}{2}\cdot0 + \frac{1}{2} \left(q\cdot 3+ (1-q)\left(\frac{1}{2}V_{p,q}(a)+\frac{1}{2}\cdot0\right) \right)\right] \Rightarrow \\
V_{p,q}(a) &= \frac{p+\frac{3}{2}(1-p)q}{1-\frac{1}{4}(1-p)(1-q)}, \enskip \textrm{ and}\\
V_{p,q}(b) & = q\cdot 1 + 
(1-q)\left[\frac{1}{2}\cdot0 + \frac{1}{2} \left(p\cdot 0+ (1-p)\left(\frac{1}{2}V_{p,q}(b)+\frac{1}{2}\cdot 4\right) \right)\right] \Rightarrow \\
V_{p,q}(b) &= \frac{q+(1-p)(1-q)}{1-\frac{1}{4}(1-p)(1-q)}.
\end{align*}
Choose $p'=\frac{1}{5}$ and $q'=\frac{3}{5}$, i.e. consider the mixed strategy stopping time $\tau_{\frac{1}{5},\frac{3}{5}}$. The corresponding expected values are $V_{\frac{1}{5},\frac{3}{5}}(a)=V_{\frac{1}{5},\frac{3}{5}}(b)=1$. All we need to do in order to verify that $\tau_{\frac{1}{5},\frac{3}{5}}$ is a mixed strategy equilibrium stopping time, is to check that neither agent $a$ nor agent $b$ wants to deviate from it, i.e. we need to verify that $V_{p,\frac{3}{5}}(a)\leq V_{\frac{1}{5},\frac{3}{5}}(a) =1$ for all $p\in[0,1]$ and that $V_{\frac{1}{5},q}(b)\leq V_{\frac{1}{5},\frac{3}{5}}(b) =1$ for all $q\in[0,1]$. This is easily done as in fact $V_{p,\frac{3}{5}}(a) = V_{\frac{1}{5},q}(b)=1$ for all $p,q\in[0,1]$. It follows that the stopping time $\tau_{\frac{1}{5},\frac{3}{5}}$ is indeed a mixed strategy equilibrium stopping time.
\end{ex}

\section{A forward iteration approach}\label{sec:forward_iteration}
The previous example illustrates that there is no hope to come up with a general method to find equilibrium stopping times {(of the pure Markov strategy type, see Definition \ref{def:equ_stop_time})}. In particular cases, this can however be done.
We now propose an approach for constructing a candidate for an equilibrium stopping time by solving a --- possibly terminating --- sequence of ordinary optimal stopping problems. More precisely, we construct a set $\hat S$ and prove that --- under certain
 assumptions --- the first entrance time $\tau_{\hat S}$ into $\hat S$
   is an equilibrium stopping time. 

To this end, write
\[S_0:=\emptyset,\;\;v_0(x,y):=\sup_{\tau}\E_x(e^{-r\tau}F(X_\tau,y)).\]
and define recursively for all $n\geq 1$
\begin{align*}
S_n&:=\{x\in \state:v_{n-1}(x,x)=F(x,x)\},\\
v_n(x,y)&:=\sup_{\tau\leq \tau_{S_n}}\E_x(e^{-r\tau}F(X_\tau,y)).
\end{align*}
Note that\ $v_n(\cdot,y)$ is the value function of an ordinary optimal stopping problem for the process $X$ absorbed in $S_n$.  
It holds that $S_1,S_2,...$ is an increasing sequence of sets\ and we assume that
\begin{align}	\label{eq:A1}
S_1,S_2,\dots\mbox{ are closed sets.}\tag{A1}
\end{align}

We denote the closure of the union $\bigcup_{n=0}^\infty S_n$ {in $E$} by $\hat S$. Moreover, $v_n$ is decreasing in $n$ and therefore converges to a limit $v_\infty$. By the construction of the problem, it is furthermore natural to assume that
\begin{align}	\label{eq:A2}
v_\infty(x,x)=\sup_{\tau\leq \tau_{\hat S}}\E_x(e^{-r\tau}F(X_\tau,x))\mbox{ for all }x\in \state\tag{A2}.
\end{align}

Our candidate for the equilibrium stopping time is now the first entrance time $\tau_{\hat S}$ into $\hat S$. The heuristic motivation is as follows: In case it is rational for the agent in state $y=x$ to stop immediately in the starting state $X_0=x$ in problem $v_n(x,x)$, $n$ minimal, say, there is no reason for her not to stop immediately in $x$ under the global time $\tau_{\hat S}$ as $\tau_{\hat S}\leq \tau_{S_n}$. Hence, the agent should accept $\tau_{\hat S}$ when $x\in \hat S$. \\
On the other hand, in the case $x\not\in \hat S$, there exists a stopping time $\tau\leq \tau_{\hat S}$ that gives strictly more expected reward than to stop immediately. In case the structure of the problem is such that  
\begin{align}
\mbox{ \eqref{eqdef2}  is satisfied with $\hat\tau=\tau_{\hat S}$ for all }x\in {\hat S\backslash\bigcup_{n\in\N}S_n}
,\tag{A3}\label{eq:A4}
\end{align}
and
\begin{align}
F(x,x)\leq \E_x(e^{-r\tau_{\hat S}}F(X_{\tau_{\hat S}},x))\mbox{ for all }x\not\in \hat S,\tag{A4}\label{eq:A3}
\end{align}
we see that it is also in this case optimal for the agent to accept $\tau_{\hat S}$. Indeed:

\begin{thm}\label{thm:forward}
Under the assumptions \eqref{eq:A1} -- \eqref{eq:A3}, the stopping time $\tau_{\hat S}$ defined above is an equilibrium stopping time.
\end{thm}

\begin{proof} Write $\hat\tau=\tau_{\hat S}$ for short. Let us first consider $x\not\in \hat S$. As $\hat S$ is closed, we find $h_0>0$ such that the open ball $B(x,h_0)$ around $x$ with radius $h_0$ is a subset of $\hat S^c$. Therefore, 
	\[J_{\hat\tau}(x)=J_{\hat\tau\circ \theta_{\tau_h}+\tau_h}(x)\]
	for all\ $h\leq h_0$, so that \eqref{eqdef2} is fulfilled automatically. Furthermore, \eqref{eq:A3} warrants \eqref{eqdef1}.
	
	For $x\in {\hat S\backslash\bigcup_{n\in\N}S_n}$, \eqref{eqdef1} holds trivially as $\hat\tau$ calls for immediate stopping, and \eqref{eq:A4}
yields \eqref{eqdef2}.
	
	It remains to check that the equilibrium conditions \eqref{eqdef1} and \eqref{eqdef2} are fulfilled for $x\in \bigcup_{n\in\N}S_n$.
		{	In this case \eqref{eqdef1} holds trivially.}
	For the second property, 	
	find $n\in\N$ such that $x\in S_n\setminus S_{n-1}$.
As $S_{n-1}$ is closed, there exists $\epsilon_0>0$ such that the ball $B(x,\epsilon_0)$ around $x$ with radius $\epsilon_0$ is a subset of $S_{n-1}^c$. Then, for each $h<\epsilon_0$ it holds that $\hat\tau\circ \theta_{\tau_h}+\tau_h\leq \tau_{S_{n-1}}$ and therefore
	\begin{align*}
	F(x,x)&=v_n(x,x)=\sup_{\tau\leq \tau_{S_{n-1}}}\E_x(e^{-r\tau}F(X_\tau,x))\\
	&\geq \E_x(e^{-r{\hat\tau\circ \theta_{\tau_h}+\tau_h}}F(X_{{\hat\tau\circ \theta_{\tau_h}+\tau_h}},x)) = J_{\hat\tau\circ \theta_{\tau_h}+\tau_h}(x),
	\end{align*}
where we used that $x\in S_n$ implies $v_n(x,x)=F(x,x)$ and  $v_{n-1}(x,x)=F(x,x)$. This yields 
		\[\liminf_{h\searrow 0}\frac{J_{\hat\tau}(x)-J_{\hat\tau\circ \theta_{\tau_h}+\tau_h}(x)}{\E_x(\tau_h)}=\liminf_{h\searrow 0}\frac{F(x,x)-J_{\hat\tau\circ \theta_{\tau_h}+\tau_h}(x)}{\E_x(\tau_h)}\geq 0.\]
\end{proof}

\begin{rem} \label{remark-on-assumptions} We now discuss the assumptions above. 
	\begin{itemize}
		\item On \eqref{eq:A1}: Optimal stopping sets are well-known to be closed under weak assumptions, see \cite{peskir2006optimal}, I.2.2. In particular, \eqref{eq:A1} is obviously fulfilled if 
		\begin{align*}
		x\mapsto F(x,x),\;x\mapsto v_n(x,x)\mbox{ are continuous.}
		\end{align*}
		\item On \eqref{eq:A2}: This assumption warrants that --- in the sense described above --- the optimal stopping sets of the problems related to\ $v_n$ converge to the optimal stopping set of the limiting problem. In particular, if the procedure terminates, i.e., there exists $n_0\in\N$ such that $S_{n_0}=S_{n_0+1}$, then assumption \eqref{eq:A2} is automatically fulfilled. 
		\item On \eqref{eq:A4}: This assumption is trivially fulfilled when the procedure terminates. In general, it can be understood as a version of a smooth fit property for the limiting problem. 
		\item On \eqref{eq:A3}: In contrast to the previous conditions, \eqref{eq:A3} is more than a technical regularity assumption. As mentioned above, it is by construction clear that for $x\not\in \hat S$, there exists a stopping time $\tau\leq \tau_{\hat S}$ with strictly larger expected reward than to stop immediately. But it is not clear in general that $\tau_{\hat S}$ also has this property. As discussed at the end of this section, Example \ref{ex:counter} is a counterexample.
	\end{itemize}
\end{rem}

\begin{rem}\label{rem:}
In some cases of interest, for example in the  one-dimensional case of Section \ref{sec:one-sided}, the procedure terminates already after one step, i.e.
\[\hat S=\{x\in\ E:F(x,x)=\sup_{\tau}\E_x(e^{-r\tau}F(X_\tau,x)) \}. \]
Under \eqref{eq:A1} it holds that $X_{\hat\tau}\in \hat S$, where we write $\hat\tau=\tau_{\hat S}$ as above. Hence, in the case of termination after one step, we obtain, using the strong Markov property, for all $x\in\ E$ and all stopping times $\sigma$,  
\[\mathbb{P}_x\left(\E_{X_{\hat\tau}}\left(e^{-r\sigma}F(X_\sigma,X_0) \right) >F(X_{\hat\tau},X_{\hat\tau})\right)=0.\]
This may be interpreted as an adaptation of the notion of dynamic optimality, see \cite{pedersen2016optimal} (and also Section \ref{sec:prev-lit}), to our setup. Hence, in this case the equilibrium (a local property) is in this sense also dynamically optimal (a global property). We remark that the equilibrium in Example \ref{optimistic-American} below is not dynamically optimal.
A (trivial) sufficient condition for $S_1=S_2$, i.e. for the procedure to terminate after one step, in the general case, is that: for all $y \notin S_1$ it holds that $S_1 \subseteq S_y$, where $S_{y}$ is defined as the stopping set for the standard stopping problem $v_0(x,y)=\sup_{\tau}\E_x(e^{-r\tau}F(X_\tau,y))$ where $y$ is fixed. To see this note that in this case, if $x \in S_2\backslash S_1$  then: (a) $v_1(x,x)=F(x,x)$ (by definition of $S_2$ and $x\in S_2$) and, (b) 
$v_0(x,x)= \sup_{\tau}\E_x(e^{-r\tau}F(X_\tau,x)) = \sup_{\tau\leq \tau_{S_1}}\E_x(e^{-r\tau}F(X_\tau,x)) = v_1(x,x)$ (to see this use that $x\notin S_1$ and the sufficient condition, which imply that the restriction $\tau\leq \tau_{S_1}$ is not relevant for optimality). From (a) and (b) follows that $v_0(x,x)=F(x,x)$ which implies that $x\in S_1$ and we have thus reached a contradiction. Hence, no $x$ satisfying  $x \in S_2\backslash S_1$ exists, i.e. $S_2\subseteq S_1$. Moreover, since $\{S_n\}$ is an increasing sequence it holds that $S_1\subseteq S_2$, and the claim follows. A more interesting sufficient condition for the one-dimensional case is provided in Theorem \ref{thm:one-sided-equ}, see also the proof.
\end{rem}

We close this section by discussing two examples. A general class of examples with a one-sided equilibrium stopping time found by this approach is discussed separately in Section \ref{sec:one-sided}.

\begin{ex} \label{optimistic-American}
	We now consider an underlying one-dimensional  Wiener process $X$ and fix a discount rate $r>0$. To illustrate the theory with an explicit example, we look at the reward function
	\[F(x,y):=\begin{cases}
	x^+&,\;y\geq 0,\\
	(-x)^+&,\;y<{0}.
	\end{cases}
	\]
	A (somewhat artificial) financial interpretation is that the holder of a perpetual American option with strike $0$ in a Bachelier market is uncertain whether she has bought a put or a call option. She is inherently optimistic and changes her belief depending on the state the process is in. If 
	 the current state is non-negative, i.e. $y\geq0$, she believes that the derivative is a call, and a put otherwise. 
	Using standard approaches to the solution of optimal stopping problems, such as a free boundary approach or the harmonic function technique of \cite{CI11}, it is straightforward to find that for fixed $y\geq0$
	\[v_0(x,y)=\sup_{\tau}\E_x(e^{-r\tau}F(X_\tau,y))=\begin{cases}
	x&,\;x\geq x_1,\\
	a_1e^{cx}&,\;x<x_1,
		\end{cases}\]
	where $c=\sqrt{2r}$, $x_1=1/c$ and $a_1=1/(ec)$. Due to symmetry, we have for $y<0$ that $v_0(x,y)=v_0(-x,-y).$ Therefore,
	\[S_1=(-\infty,-x_1]\cup[x_1,\infty).\]
	Now, we can go on iteratively to find $v_n$ and $S_n$ again using standard arguments for optimal stopping problems for diffusions. Writing 
	\[f_n(x)=a_ne^{-cx}+b_ne^{cx},\]
	we try to find $x_{n+1},\,a_n,\,b_n$ such that
	\begin{align*}
	f_n(-x_{n-1})&=0,\\
	f_n(x_{n})&=x_n,\\
	f_n'(x_{n})&=1.
	\end{align*}
	This system is indeed solvable and the solution is given by
	\begin{align*}
a_n&=\frac{1}{2}e^{cx_n}\left(x_n-\frac{1}{c}\right),\\
b_n&=\frac{1}{2}e^{-cx_n}\left(x_n+\frac{1}{c}\right)
\end{align*}	
and $x=x_n$ is the unique solution in $(0,x_{n-1})$ of
\[e^{2cx}=e^{-2cx_{n-1}}\frac{\frac{1}{c}+x}{\frac{1}{c}-x}.\]
Then, 
	\[S_n=(-\infty,-x_n]\cup[x_n,\infty)\]
and for $y\geq0$
\[v_n(x,y)=\begin{cases}
x&,\;x\geq x_{n+1},\\
f_n(x)&,\;-x_{n+1}\leq x<x_{n+1},\\
0&,\;x<-x_{n+1}
\end{cases}\]
and, as above, for $y<0$ it holds that $v_n(x,y)=v_n(-x,-y).$
It is easily seen that $x_n$ converges monotonically to the unique solution $x=x^*$ in $(0,1/c)$ of
\[e^{4cx}=\frac{\frac{1}{c}+x}{\frac{1}{c}-x},\]
so that
\[\hat S=(-\infty,-x^*]\cup[x^*,\infty)\]
and for $y\geq0$
\[v_\infty(x,y)=\begin{cases}
x&,\;x\geq x^*,\\
f_\infty(x)&,\;-x^*\leq x<x^*,\\
0&,\;x<-x^*
\end{cases}\]
with 	
	\[f_\infty(x)=\frac{1}{2}e^{cx^*}\left(x^*-\frac{1}{c}\right)e^{-cx}+\frac{1}{2}e^{-cx^*}\left(x^*+\frac{1}{c}\right)e^{cx}.\]
Again, $v_\infty(x,y)=v_\infty(-x,-y)$ for $y<0$.
 
			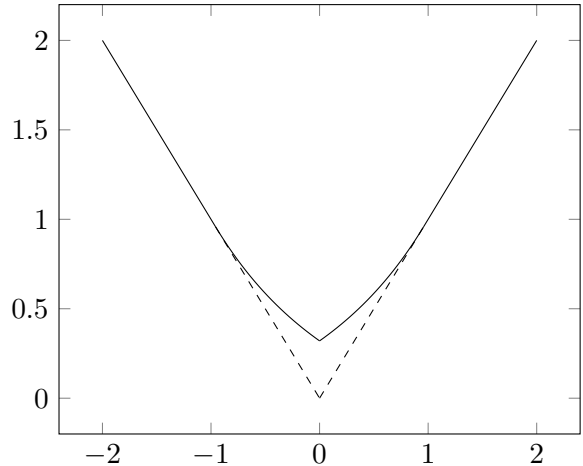
\begin{figure}[ht]
	\begin{center}
	\begin{tikzpicture} 
	\begin{axis} 
	\addplot[domain=0.9575:2] {x}; 
		\addplot[style=dashed,domain=0:0.9575] {x}; 
	\addplot[style=solid,domain=0:0.9575] {(1/2)*exp(1*0.9575)*(0.9575-1/1)*exp(-x)+(1/2)*exp(-1*0.9575)*(0.9575+1/1)*exp(x)}; 
	\addplot[style=solid,domain=-0.9575:0] {(1/2)*exp(1*0.9575)*(0.9575-1/1)*exp(x)+(1/2)*exp(-1*0.9575)*(0.9575+1/1)*exp(-x)}; 
			\addplot[style=dashed,domain=-0.9575:0] {-x}; 
		\addplot[domain=-2:-0.9575] {-x}; 
	\end{axis} 
	\end{tikzpicture}   
		\end{center}
			\caption{The functions $x\mapsto v_\infty(x,x)$ and $x\mapsto F(x,x)$ for $c=1$. Here, $x^*\approx 0.9575$.}
\end{figure}
It is straightforwardly verified that \eqref{eq:A1} and \eqref{eq:A2} are fulfilled. \eqref{eq:A4} holds by the smoothness of the function $v_\infty$ in\ $x^*$. \eqref{eq:A3} could be verified using the theory developed in the following section. Here, it is however immediately checked elementary due to the convexity of $f_\infty$ for $x\geq0$ and $f_\infty'(x^*)=1$. Hence, Theorem \ref{thm:forward} yields that 
\[\hat\tau=\inf\{t\geq 0: |X_t|\geq x^*\}\]
is an equilibrium stopping time. 

\end{ex}

\begin{ex}\label{ex:counter} We now come back to Example \ref{mixedSec}. We already know that there is no pure Markov strategy equilibrium stopping time, so that the approach described in this section cannot be successful. Indeed, $S_1=\{\partial_1,\partial_2\}$ and the procedure terminates after this step. As argued in Example \ref{mixedSec} \ref{item:ex_(iv)}, this is no equilibrium stopping time. More precisely, condition \eqref{eq:A3} fails to hold true.
\end{ex}

\section{A class of one-sided solvable problems with potential jumps}\label{sec:one-sided}
As a more advanced application of the method in the previous example, we consider a  Markov process on the real line. 
To construct an equilibrium stopping time for a wide class of examples, we consider the general setting of \cite{CST} for the auxiliary optimal stopping problems with value function
\begin{equation}\label{OSP_aux}
v_y(x)=\sup_{\tau}\E_x(e^{-r\tau}F(X_\tau,y)),\;y\in\R.
\end{equation}
That is, we assume that each function $F(\cdot,y)$ has a representation of the form
\begin{equation}\label{eq:M_T_repr}
F(x,y)=\E_x\big({Q_y}(M_T)\big),
\end{equation}
where $M_t:=\sup_{0\leq s\leq t}X_s,\;t\geq0,$ denotes the running maximum process of $X$ and $T$ is an exponentially with parameter $r$ distributed
random variable independent of $X$. In this section we assume that $r>0$. At first sight, it is not clear at all why such a representation should exist. However, as detailed in Section 2.2 of \cite{CST}, it always exists under suitable integrability and smoothness assumptions. More explicitly, it is given by
\begin{align}\label{conditional_dist}
&{Q_y}(z):=\frac{1}{r}\int_{-\infty}^{z}(r-A_X)F(u,y)\mathbb{P}_{x}(X_T\in du|M_T=z),\\
&\mathbb{P}_{x}(X_T\in du| M_T=z):=\mathbb{P}_{x}(X_T\in du\,,\, M_T\in
dz)/\mathbb{P}_{x}(M_T\in dz),\nonumber
\end{align}
where $A_X$ denotes the (extended) infinitesimal generator of $X$.
\begin{rem} If $A_X$ is applied to a function $E\times E\rightarrow \mathbb{R}$ then $A_X$ should, here and in the following, be understood to only act on the first variable.
\end{rem}

The conditional density used above can be found (semi-)explicitly for general L\'evy processes and diffusions, so that also $Q_y$ is given in analytical terms in these cases. The following result, which follows directly from Theorem 2.5 in \cite{CST}, then leads to the solution of the auxiliary optimal stopping problems in case they are of a one-sided form:

\begin{lem}
	\label{thm:one-sided}
		Assume that for each $y$ there exists a point $x^*_y$ such that
		\begin{enumerate}[label=(B\arabic*)]
			\item \label{item:B1} $Q_y(x)\leq 0$ for $x\leq x^*_y$,
			\item\label{item:B2} $Q_y(x)$ is positive and  non-decreasing for $x>x^*_y$.
		\end{enumerate}
	Then, the value function of the auxiliary optimal stopping problem \eqref{OSP_aux} is given by
	\begin{equation}
	\label{value_fct}
	v_y(x)=\E_x\left(Q_y(M_T)1_{\{M_T\geq x^*_y\}}\right)
	\end{equation}
	and
	\[\tau^*_y:=\inf\{t\geq 0: X_t\geq x^*_y\}\]
	is an optimal stopping time.
\end{lem}
 
Now, using the approach described in Section \ref{sec:forward_iteration}, we obtain the following verification theorem for problems where the underlying auxiliary optimal stopping problems are one-sided.

\begin{thm}\label{thm:one-sided-equ}
In the one-dimensional setting of this section assume that for each $y$ there exists a point $x^*_y$ such that \ref{item:B1} and \ref{item:B2} hold true. Furthermore, assume that there exists a point $x^*$ such that $x^*_{y}\leq y$ for $y\geq x^*$ and $x^*_{y}\geq x^*$ for $y\leq x^*$. Then
\[\hat \tau:=\inf\{t\geq 0:X_t\geq x^*\}\]
	is an equilibrium stopping time. 
\end{thm}
\begin{rem} 
The conditions of Theorem \ref{thm:one-sided-equ} imply that if the function $y \mapsto x^*_y$ is continuous then $x^*$ is the unique fixed point of that function.
\end{rem}
\begin{proof}  Note that Lemma \ref{thm:one-sided} yields that the forward iteration sequence of Section \ref{sec:forward_iteration} is given by
\[S_1=[x^*,\infty)\]
and the procedure then terminates i.e. $S_1=S_2=S_3=...=\hat S$ and $S_1$ is closed (this can easily be seen directly and it also follows from the following argument). To apply Theorem \ref{thm:forward}, it remains to check \eqref{eq:A3}, i.e.
\begin{align*}
F(x,x)\leq \E_x(e^{-r\hat\tau}F(X_{\hat \tau},x))\mbox{ for all }x<x^*.
\end{align*}
This, however, holds as
\begin{align*}
F(x,x)&=\E_x\big({Q_x}(M_T)\big)\leq \E_x\big({Q_x}(M_T)1_{\{M_T\geq x^*\}}\big),
\end{align*}
where we used that ${Q_x}(M_T)$ is non-positive on $\{M_T< x^*\} \subseteq \{M_T< x^*_x\}$ by \ref{item:B1}. We conclude by noting that Lemma 2 in \cite{christensen2015impulse} yields
\[\E_x\big({Q_x}(M_T)1_{\{M_T\geq x^*\}}\big)=\E_x(e^{-r\hat\tau}F(X_{\hat \tau},x)).\]
\end{proof}

\begin{rem} By applying the previous results to $-X$, we immediately obtain an analogous result for the case that the auxiliary optimal stopping sets are of left-sided type $(-\infty,x^*_y]$.
\end{rem}

\begin{ex} \label{state-dependent-American}To illustrate the general approach above, we consider a perpetual American call problem with state-dependent strike $K(y)$ in a general L\'evy market. One interpretation is an investor who has forgotten the concerted strike of the option. Depending on the state of the price process, she changes her opinion on the concerted strike   
--- this situation is of course typically not realistic and the example is included only in order to illustrate the theory, however, see Remark \ref{put-ver}.
 More concretely, her reward function for the log-price process $X$ has the structure
\[F(x,y)=(e^x-K(y))^+,\]
where we assume that the function $K:\R\rightarrow (0,\infty)$ is 
continuous and non-increasing; the interpretation of this is that the investor believes the strike to be lower when the asset price is higher. Let $X$ be a general L\'evy process. To avoid trivial cases, we assume $X$ not to be a subordinator and to fulfill $E_0(e^{X_1})<{e}^r$. For technical reasons, we first ignore the $(\cdot)^+$, i.e. we change the reward function to 
 \[\tilde F(x,y)=e^x-K(y),\]
 which makes some arguments and notations in the following shorter. 
Using the approach from \cite{CST}, or just by guessing, we see that the function $Q_y$ is given by
\[Q_y(x)=ae^{x}-K(y),\]
where $a={1}/{\E_0e^{M_T}}<1$, see also \cite{M}. The value of $a$ can be found more explicitly for many classes of processes. For example, for L\'evy processes without positive jumps, $M_T$ is exponentially distributed. In the case of a Wiener process $X$, we obtain $a=\frac{\sqrt{2r}-1}{\sqrt{2r}}$. The optimal stopping boundary for the auxiliary problem is therefore, by Lemma \ref{thm:one-sided}, given by $x_y^*={\log(K(y)/a)}$. Now use the properties of $K(\cdot)$ to verify that the conditions of Theorem \ref{thm:one-sided-equ} are satisfied and that there exits a (unique) fixed point
\[x^*={\log(K(x^*)/a)}.\]

It follows from Theorem \ref{thm:one-sided-equ} that the equilibrium stopping time for the reward function $\tilde F(x,y)$ is given by 
\begin{align} \label{state-dependent-American-ST}
\hat \tau=\inf\{t\geq 0:X_t\geq x^*\}.
\end{align}
We may therefore conclude that the corresponding equilibrium value function $\tilde{J}_{\hat\tau}(x):= \E_x\left(e^{-r{\hat \tau }}\tilde{F}(X_{\hat \tau },x)\right)$ and the reward function $\tilde F(x,y)$ satisfy the equilibrium properties \eqref{eqdef1} and \eqref{eqdef2}.

In order to show that \eqref{state-dependent-American-ST} is an equilibrium stopping time also for the original reward function $F(x,y)=(e^x-K(y))^+$, let us verify that also $J_{\hat\tau}(x):= \E_x\left(e^{-r{\hat \tau }}F(X_{\hat \tau },x)\right)$ and $F(x,y)$ satisfy \eqref{eqdef1} and \eqref{eqdef2}: First, note that if $x$ is such that $e^{x^*}-K(x)\geq0$, then 
$\tilde{J}_{\hat\tau}(x)= \E_x\left(e^{-r{\hat \tau }}(e^{X_{\hat \tau }}-K(x))\right) = \E_x\left(e^{-r{\hat \tau }}(e^{X_{\hat \tau }}-K(x))^+\right)=J_{\hat\tau}(x)$, and similarly $\tilde J_{\hat\tau\circ \theta_{\tau_h}+\tau_h}(x)=J_{\hat\tau\circ \theta_{\tau_h}+\tau_h}(x)$. Moreover, if $x$ is such that $e^{x^*}-K(x)<0$, then we are in the continuation region 
and hence ${\hat\tau} ={\hat\tau\circ \theta_{\tau_h}+\tau_h}$ for sufficiently small $h$. It follows that $J_{\hat\tau}(x)=J_{\hat\tau\circ \theta_{\tau_h}+\tau_h}(x)$ for sufficiently small $h$. We conclude that $J_{\hat\tau}(x)$ satisfies \eqref{eqdef2}. 
Second, note that if $x\geq x^*$ then $J_{\hat\tau}(x)=F(x,x)$ and \eqref{eqdef1} follows trivially. Let us deal with the case $x< x^*$. If $x$ is such that $e^x-K(x)<0$ then $F(x,x)=0$, and since $J_{\hat\tau}(x)\geq 0$, it follows that \eqref{eqdef1} satisfied. If $x$ is such that $e^x-K(x)\geq0$ then $e^{x^*}-K(x)>0$ which implies that $J_{\hat\tau}(x)=\tilde J_{\hat\tau}(x)$ and $\tilde F(x,x) = F(x,x)$. We conclude that $J_{\hat\tau}(x)$ and $F(x,x)$ satisfy \eqref{eqdef1}.

We have thus shown that the equilibrium stopping time for the original reward $F(x,y)=(e^x-K(y))^+$ is also given by \eqref{state-dependent-American-ST} and it follows that the corresponding equilibrium value function can be written as
\begin{equation}
	J_{\hat\tau}(x)=\begin{cases}
	e^x-K(x),&x\geq x^*,\\
	\E_x\left(e^{-r\hat\tau}\left(e^{X_{\hat\tau}}-K(x)\right)^+\right),&x<x^*,
	\end{cases}
	\end{equation}
	where more explicitly, for $x$ with $\log K(x)<x^*$,
	\begin{align*}
	\E_x\left(e^{-r\hat\tau}\left(e^{X_{\hat\tau}}-K(x)\right)^+\right)
	&=\E_x\left(e^{-r\hat\tau}\left(e^{X_{\hat\tau}}-K(x)\right)\right)\\
	&=a\E_x\left(e^{M_T}1_{\{M_T\geq x^*\}}\right)-K(x)\mathbb P_x(M_T\geq x^*),
	\end{align*}
\end{ex}

\begin{rem} \label{put-ver} {A put-version of Example \ref{state-dependent-American} can be interpreted --- economically more meaningful --- as an equilibrium selling problem under endogenous habit formation and exponential utility in a Bachelier market. That problem is, however, analyzed  and discussed in the more realistic Black-Scholes market in Example \ref{eq-selling-strat}.}
\end{rem}

\section{The time-inconsistent variational inequalities}\label{sec:time-incon-var-eq} 
In the rest of the paper we assume that the state process $X$ is the strong solution to the $d$-dimensional SDE
\begin{align} \label{SDE}
dX_t = \mu(X_t)dt + \sigma(X_t)dW_t, \enskip X_0=x\in E,
\end{align} 
{where $W$ is an $r$-dimensional Wiener process, the state space $E\subseteq\mathbb{R}^d$ is an open set, and the deterministic functions $\mu$ and $\sigma$ are continuous. Standard conditions for the existence of a strong solution to \eqref{SDE} can be found in e.g. \cite{karatzas2012brownian}. Note that we do not exclude the possibility that $E = \mathbb{R}^d$. The generator $A_X$ is now given by the differential operator}
\[A_X = \sum_i^d\mu_i(x) \frac{ \partial}{\partial x_i} + \frac{1}{2}\sum_{i,j}^da_{i,j}(x)\frac{ \partial^2}{\partial x_ix_j}, \enskip a(x):=\sigma(x)\sigma^T(x).\]

\subsection{A heuristic derivation of the time-inconsistent variational inequalities}
In this subsection we heuristically derive the time-inconsistent variational inequalities. We remark that this section is only of motivational value and that there are no claims of rigor in the derivation. In this {sub}section we consider $r=0$ for the ease of exposition. Suppose an equilibrium stopping time $\hat\tau$ exists, see Definition \ref{def:equ_stop_time}. Recall that $\tau_h=\inf\{t\geq 0:|X_t-X_0|\geq h\}$ and let $f_{\hat\tau}(X_{\tau_h},x)$ denote the auxiliary function that uses the equilibrium stopping time given the starting value $X_{\tau_h}$. Given sufficient regularity, we use the strong Markov property to see that
\begin{align}
J_{\hat\tau\circ \theta_{\tau_h}+\tau_h}(x) 
&:= \E_x(F(X_{\hat\tau\circ \theta_{\tau_h}+\tau_h},x))\\
&=\E_x(\E_{X_{\tau_h}}(F(X_{\hat\tau},x)))\\
& = \E_x(f_{\hat\tau}(X_{\tau_h},x)) \label{heur0}
\end{align}
and It\^o's formula to obtain
\begin{align}
\E_x(f_{\hat\tau}(X_{\tau_h},x))
= f_{\hat\tau}(x,x) +\E_x\left(\int_0^{\tau_h}A_Xf_{\hat\tau}(X_t,x)dt\right), \label{heur0.5}
\end{align}
where we recall that the differential operator $A_X$ operates only on the first variable. 
We now use the dominated convergence theorem, Lebesgue's differentiation theorem, \eqref{heur0}, \eqref{heur0.5} and $J_{\hat\tau}(x)=f_{\hat\tau}(x,x)$ (cf. Definition \ref{def:equ_func}) to obtain, 
under sufficient regularity, 
\begin{align}
\liminf_{h\searrow 0}\frac{J_{\hat\tau}(x)-J_{\hat\tau\circ \theta_{\tau_h}+\tau_h}(x)}{\E_x(\tau_h)} 
& = \liminf_{h\searrow 0}\frac{f_{\hat\tau}(x,x)-\E_x(f_{\hat\tau}(X_{\tau_h},x))}{\E_x(\tau_h)}\\
& = \liminf_{h\searrow 0}\frac{\E_x\left(\int_0^{\tau_h}-A_Xf_{\hat\tau}(X_t,x)dt\right)}{\E_x(\tau_h)}\\
& =		- A_Xf_{\hat\tau}(x,x).
\end{align}
This, of course, reflects the well-known characterization of the infinitesimal generator due to Dynkin. The definition of an equilibrium stopping time in Definition \ref{def:equ_stop_time} therefore translates to $A_Xf_{\hat\tau}(x,x)\leq 0$ and $J_{\hat\tau}(x)-F(x,x)\geq 0$. Now note that $J_{\hat\tau}(x) = f_{\hat\tau}(x,x)$ implies that the equilibrium value function $J_{\hat\tau}(x)$ is completely determined by the auxiliary function $f_{\hat\tau}(x,y)$. We therefore summarize the above in terms of the auxiliary function:
\begin{align} 
f_{\hat\tau}(x,x)\geq F(x,x),\enskip x \in E\label{heur1}\\
A_Xf_{\hat\tau}(x,x)\leq 0, \enskip x \in E.\label{heur2}
\end{align}
For any $x$, stopping yields the value $F(x,x)$. Using \eqref{heur1} we see that it is therefore optimal, for the $x$-agent, to stop if and only if $f_{\hat\tau}(x,x)=F(x,x)$. Suppose that the set 
\[C = \left\{x \in E: f_{\hat\tau}(x,x) > F(x,x) \right\}\]
is open, where $C$ is said to be the continuation region. It follows that the corresponding equilibrium stopping time is the first exit time from $C$, or analogously the first entrance time into the stopping region $E \backslash C$, i.e
\begin{align} \label{EQstoppingtime}
\tau_{E \backslash C} = \inf\{t \geq 0: X_t \in E \backslash C\} = \inf\{t \geq 0: X_t \notin C\},
\end{align}
which implies that $f_{\hat\tau}(x,y)=\E_x(F(X_{\tau_{E \backslash C}},y))$ for all $x$ and $y$. By \eqref{EQstoppingtime} it is also clear that if $x\in E \backslash C$, then $\tau_{E \backslash C}=0$ which implies that $f_{\hat\tau}(x,y)=\E_x(F(X_{\tau_{E \backslash C}},y))=F(x,y)$. It therefore holds, for any $y$, that 
\begin{align*}
f_{\hat\tau}(x,y) = F(x,y), \enskip x \in E \backslash C.
\end{align*}

Moreover, since $f_{\hat\tau}(x,y)=\E_x(F(X_{\tau_{E \backslash C}},y))$ it follows that $f_{\hat\tau}(X_t,y)$ is a martingale on $C$,  for any fixed $y$, given sufficient regularity. Hence, for any fixed $y$, 
\begin{align*}
A_Xf_{\hat\tau}(x,y)=0, \enskip x\in C.
\end{align*}
Let us summarize our findings. If an equilibrium stopping time exists then, under the assumption of sufficient regularity, it is given by $\tau_{E \backslash C}$ defined in \eqref{EQstoppingtime} and the auxiliary function 
$f_{\hat\tau}(x,y)=\E_x(F(X_{\tau_{E \backslash C}},y))$ satisfies
\begin{align}
A_Xf_{\hat\tau}(x,x)\leq 0, \enskip x \in E,
\end{align}
and for any fixed $y\in E$
\begin{align}
A_Xf_{\hat\tau}(x,y) &=0, \enskip x \in C,\\
f_{\hat\tau}(x,y) -F(x,y) &= 0, \enskip x \in E \backslash C,
\end{align}
where
\begin{align}
C = \left\{x\in E: f_{\hat\tau}(x,x) > F(x,x) \right\}.
\end{align}
We call the expressions above the \emph{time-inconsistent variational inequalities}.

\subsection{A verification theorem}
Let us define the time-inconsistent variational inequalities in more detail.
\begin{defi} \label{time-incon-free-B} A function 
	$f:E\times E\rightarrow\mathbb{R}$ is said to satisfy the \emph{time-inconsistent variational inequalities} if\footnote{Recall that the differential operator $A_X$ operates only on the first variable, in e.g. $f(x,x)$.}
\begin{align}
A_Xf(x,x)-rf(x,x)&\leq 0, \enskip x\in E\backslash \partial C,\label{freeB1}
\end{align}
and, for each fixed $y\in E$,
\begin{align}
A_Xf(x,y) - rf(x,y)&=0, \enskip x\in C,\label{freeB3}\\
f(x,y)-F(x,y) &= 0, \enskip x\in  E \backslash C,\label{freeB3-2}
\end{align}
where
\begin{align}
C:= \left\{x \in E: f(x,x) > F(x,x) \right\}.
\end{align}
Moreover, the function $f(\cdot,y):E \rightarrow \mathbb{R}$ must, for each fixed $y\in E$, satisfy:
\begin{enumerate}[label=(\roman*)]

\item \label{regularity-verTHM-onC} $f(\cdot,y) \in \mathcal{C}(\overline{C})\cap \mathcal{C}^2(C)$, where $\overline{C}$ denotes the closure of $C$ in $E$,

\item \label{regularity-verTHM2} $f(\cdot,y) \in \mathcal{C}^1(B(y,\epsilon)) \cap \mathcal{C}^2(B(y,\epsilon) \backslash \partial C)$ 
for some $\epsilon>0$, where the second order derivative is locally bounded (near $\partial C$),

\item \label{regularity-verTHMbounded} $f(\cdot,y)$ is bounded on $\overline{C}$.
\end{enumerate}
Lastly, we also demand that:
\begin{enumerate}[resume,label=(\roman*)]
\item \label{regularity-verTHMLipschitz} $C$ is   
open  and $\partial C \neq \emptyset $ is a Lipschitz surface\footnote{For a definition see \cite[~ch. 10]{oksendal2013stochastic}.}.
\item \label{cond5} 
$\limsup_{z\notin \partial C \rightarrow y}(A_Xf(z,y)-rf(z,y))\leq 0$, for $y\in \partial C$.
\end{enumerate}
\end{defi}

\begin{thm} \label{verificTHM} Suppose that a function $f:E\times E\rightarrow\mathbb{R}$ solves the time-inconsistent variational inequalities. Suppose that the state process $X$ that solves the SDE \eqref{SDE} spends almost no time on the boundary $\partial C$, i.e.
\begin{equation} \label{regularity-verTHM}
\int_0^{\infty} I_{\partial C}(X_t)dt = 0 \enskip \textrm{a.s.,}
\end{equation}
and that 
\begin{align} \label{VerTHM1}
\hat\tau:= \inf\{t\geq 0: X_t\notin C\}<\infty \enskip \textrm{a.s.},
\end{align}
for each starting value $X_0=x\in E$. 
Then, 
\begin{itemize}
\item {$J:E\rightarrow\mathbb{R}$, with $J(x):=f(x,x)$}, is an equilibrium value function,
\item ${f}:E\times E \rightarrow\mathbb{R}$ is the corresponding auxiliary function,  and 
\item the stopping time $\hat\tau$ in \eqref{VerTHM1} is the corresponding equilibrium stopping time.
\end{itemize}
\end{thm} 

\begin{proof} Recall that the state space $E \subseteq \mathbb{R}^d$ is here assumed to be an open set and note that $E$ can here, without loss of generality, be taken to be connected, since $X$ has continuous sample paths. Let $\{C_k\}_{k=1}^\infty$ be an increasing sequence of open, bounded and connected sets with $\overline{C}_k\subseteq C$ and $\cup_{k=1}^\infty C_k=C$. Consider arbitrary $y\in E$ and $x\in C$, which implies that $x\in C_k$ for {any $k\geq k'$, for some $k'$}. Let $\tau_k=\inf\{t\geq0:X_t \notin C_k\}\wedge k$. Use \ref{regularity-verTHM-onC}, It\^o's formula and \eqref{freeB3} to obtain
\begin{align} 
f(x,y) &= \E_x\left(e^{-r{\tau_k}}f(X_{\tau_k},y) - \int_0^{\tau_k}e^{-rt}(A_Xf(X_t,y)-rf(X_t,y))dt\right)\\
&=\E_x\left(e^{-r{\tau_k}}f(X_{\tau_k},y)\right)
\end{align}
where the It\^o integral has vanished by the continuity of $\sigma(x)$, the continuity of the trajectories of $X$, the continuity of $\frac{\partial^2 f(x,y)}{\partial x^2}$ on $C$, and the boundedness of $X_s$ on the bounded stochastic interval $[0,\tau_k]$. Note that \eqref{freeB3-2}, \ref{regularity-verTHM-onC} and \ref{regularity-verTHMLipschitz} imply, for fixed $y\in E$, that $f(\cdot,y)$ is continuous on $\overline{C}$ with $f(x,y)=F(x,y)$ on $\partial C$, where $\partial C\neq \emptyset$ by assumption. Since, $f(\cdot,y)$ is bounded on $\overline{C}$ (cf. \ref{regularity-verTHMbounded}) we may thus use the bounded convergence theorem, and that $\tau_k \rightarrow \hat\tau$  a.s. as $k\rightarrow\infty$ (cf. \eqref{VerTHM1}), to obtain
\begin{align} \label{verproof-okt}
f(x,y) 
= \lim_{k\rightarrow\infty}\E_x\left(e^{-r{\tau_k}}f(X_{\tau_k},y)\right)
= \E_x\left(e^{-r{\hat\tau}}F(X_{\hat\tau},y)\right).
\end{align}
Using \eqref{freeB3-2} and \eqref{VerTHM1} we see that this implies that
\begin{align} \label{Ver1}
f(x,y) = \E_x(e^{-r\hat\tau}F(X_{\hat\tau},y)) \enskip \textrm{on }E\times E.
\end{align}
Definition \ref{def:equ_stop_time}, Definition \ref {def:equ_func} and \eqref{Ver1} yield the following: if we can prove that 
\begin{align} 
&f(x,x)-F(x,x)\geq 0, \enskip \textrm{for each $x\in E$, and} \label{VerThmEqCond1}\\
&\liminf_{h\searrow 0}\frac{f(x,x)-J_{\hat\tau\circ \theta_{\tau_h}+\tau_h}(x)}{\E_x(\tau_h)}\geq 0, \enskip \textrm{for each $x\in E$,} \label{VerThmEqCond2}
\end{align}
then it follows that ${\hat\tau}$ in \eqref{VerTHM1} is an equilibrium stopping time, that ${J(x):=}f(x,x)$ is the corresponding equilibrium value function and that $f(x,y)$ is the corresponding auxiliary function. Thus, all we have left to do is to show that \eqref{VerThmEqCond1} and \eqref{VerThmEqCond2} are satisfied.

Since $f(x,y)$ solves the time-inconsistent variational inequalities we know that $f(x,x)>F(x,x)$ on $C$ and $f(x,x)=F(x,x)$ on $E \backslash C$. It follows that \eqref{VerThmEqCond1} holds.

For each fixed $y\in E$, observe that $B(y,\epsilon)$ and $B(y,\epsilon)\cap C$ are open, $B(y,\epsilon)\backslash \partial (B(y,\epsilon)\cap C) = B(y,\epsilon)\backslash \partial C$ and that $\partial (B(y,\epsilon)\cap C)$ is a Lipschitz surface (cf. \ref{regularity-verTHMLipschitz}). It therefore follows from \ref{regularity-verTHM2} and Theorem D.1 in \cite[~App. D]{oksendal2013stochastic} that there exists, for each fixed $y\in E$ and some $\epsilon>0$, a sequence of functions $\{f_i(\cdot,y)\}_{i=1}^\infty$ such that 
\begin{enumerate}[label=(\alph*)]
\item \label{OksendalApp1} $f_i(\cdot,y) \in \mathcal{C}(\overline{B}(y,\epsilon))\cap \mathcal{C}^2(B(y,\epsilon))$ for each $i$,

\item \label{OksendalApp2} $f_i(\cdot,y)$ converges to $f(\cdot,y)$ uniformly on compact subsets of $\overline{B}(y,\epsilon)$ as $i \rightarrow \infty$,

\item \label{OksendalApp3} $A_Xf_i(\cdot,y)$ converges to $A_Xf(\cdot,y)$ uniformly on compact subsets of $B(y,\epsilon)\backslash \partial C$ as $i \rightarrow \infty$,

\item \label{OksendalApp4} $\{A_Xf_i(\cdot,y)\}_{i=1}^\infty$ is locally bounded on $\overline{B}(y,\epsilon)$.
\end{enumerate} 
For any fixed $x\in E$, $h \in(0,\epsilon)$, $i$ and $k>0$ it thus follows from It\^o's formula that
\begin{align}
f_i(x,x) =
\E_x\left(e^{-r{\tau_h\wedge k}}f_i(X_{\tau_h\wedge k},x) - \int_0^{\tau_h\wedge k}e^{-rt}(A_Xf_i(X_t,x)-rf_i(X_t,x))dt\right).
\end{align}
where the It\^o integral has vanished for reasons analogous to the above. Because of the continuity in \ref{OksendalApp1} and $X_t$ being bounded on $[0,\tau_h]$ we can use the bounded convergence theorem to obtain
\begin{align}
&f_i(x,x) = \lim_{k\rightarrow \infty} f_i(x,x)\\
&= \E_x\left(\lim_{k\rightarrow \infty}\left(e^{-r{\tau_h\wedge k}}f_i(X_{\tau_h\wedge k},x) - 
\int_0^{\tau_h\wedge k}e^{-rt}(A_Xf_i(X_t,x)-rf_i(X_t,x))dt\right)\right) \\
&= \E_x\left(e^{-r{\tau_h}}f_i(X_{\tau_h},x) - \int_0^{\tau_h}e^{-rt}(A_Xf_i(X_t,x)-rf_i(X_t,x))dt\right).
\end{align}
Set the undefined $\frac{\partial^2 f(x,y)}{\partial x^2}, x\in \partial C,$ to zero.  
Now use the convergence and boundedness properties in \ref{OksendalApp2}, \ref{OksendalApp3}, and \ref{OksendalApp4}, and the regularity in \eqref{regularity-verTHM}, and the bounded convergence theorem to obtain  
\begin{align} 
f(x,x) 
& = \lim_{i \rightarrow \infty} f_i(x,x) \\
& = \E_x\left(\lim_{i \rightarrow \infty}\left( e^{-r{\tau_h}}f_i(X_{\tau_h},x) - \int_0^{\tau_h}e^{-rt}(A_Xf_i(X_t,x)-rf_i(X_t,x))dt\right)\right)\\
& = \E_x\left(e^{-r{\tau_h}}f(X_{\tau_h},x) - \int_0^{\tau_h}e^{-rt}(A_Xf(X_t,x)-rf(X_t,x))dt\right).\label{Ver2}
\end{align}
Now use \eqref{Ver1} and the strong Markov property to see that
\begin{align}
\E_x(e^{-r{\tau_h}}f(X_{\tau_h},x))
&= \E_x(e^{-r\tau_h}\E_{X_{\tau_h}}(e^{-r\hat\tau}F(X_{\hat\tau},x)))\\
&= \E_x(e^{-r(\hat\tau\circ \theta_{\tau_h}+{\tau_h})}F(X_{\hat\tau\circ \theta_{\tau_h}+\tau_h},x))\\
&= J_{\hat\tau\circ \theta_{\tau_h}+{\tau_h}}(x),\label{Ver3}
\end{align}
where we also relied on \ref{regularity-verTHMbounded} and \eqref{freeB3-2}. Using \eqref{Ver2} and \eqref{Ver3} we rewrite the left hand side of \eqref{VerThmEqCond2} as
\begin{align}
& \liminf_{h\searrow 0}\frac{-\E_x\left(\int_0^{\tau_h}e^{-rt}(A_Xf(X_t,x)-rf(X_t,x))dt\right)}{\E_x(\tau_h)}\\
& = - \limsup_{h\searrow 0}\frac{\E_x\left(\int_0^{\tau_h}e^{-rt}(A_Xf(X_t,x)-rf(X_t,x))dt\right)}{\E_x(\tau_h)}\label{verthm-ad1}.
\end{align}
Hence, all we have left to do in order to show that \eqref{VerThmEqCond2} is true, i.e. to conclude the proof, is to show that \eqref{verthm-ad1} is non-negative for all $x\in E$. Let us do this.

First consider an arbitrary $x \in E \backslash \partial C$. Recall that $\mu(x),\sigma(x)$ and the trajectories of $X$ are continuous. Note that $h\in(0,\epsilon)$ implies that the process $(X_t)_{0\leq t \leq \tau_h}$ with $X_0=x$ stays in $B(x,\epsilon)$. The regularity properties in \ref{regularity-verTHM2} therefore imply that the integrand in \eqref{verthm-ad1}, i.e. $e^{-rt}(A_Xf(X_t,x)-rf(X_t,x))$, is a.e. continuous in $t$ a.s. Recall that $X_s$ is bounded on $[0,\tau_h]$ when $h\in(0,\epsilon)$. Note also that that if we pick a sufficiently small $h=h(\omega)$ then the integrand in \eqref{verthm-ad1} is continuous in $t$ a.s, since we can for sufficiently small $h=h(\omega)$ avoid the issue that $\frac{\partial^2 f(x,y)}{\partial x^2}$ is arbitrarily set to $0$ at $\partial C$. It follows that we may use the bounded convergence theorem and Lebesgue's differentiation theorem to obtain that \eqref{verthm-ad1} is equal to
\begin{align} \label{freeXX}
& -(A_Xf(x,x)+rf(x,x))\geq 0, \textrm{ for } x \in E \backslash \partial C
\end{align}
where the inequality follows from \eqref{freeB1}. Now consider an arbitrary $x \in \partial C$. Replace the integrand in \eqref{verthm-ad1} with a right-continuous version (in $t$ a.s.) and  use \ref{cond5} and \eqref{regularity-verTHM} in the following way
\begin{align}
& - \limsup_{h\searrow 0}\frac{\E_x\left(\int_0^{\tau_h}e^{-rt}(A_Xf(X_t,x)-rf(X_t,x))dt\right)}{\E_x(\tau_h)}\\
& = - \limsup_{h\searrow 0}\frac{\E_x\left(\int_0^{\tau_h}\lim_{k\searrow 0} \sup_{0 < l\leq k}e^{-r(t+l)}(A_Xf(X_{t+l},x)-rf(X_{t+l},x))dt\right)}{\E_x(\tau_h)}\\
& \geq - \limsup_{z \notin \partial C\rightarrow x} \left( A_Xf(z,x)-rf(z,x)\right) \geq 0, \textrm{ for } x \in \partial C.
\end{align}
We have thus shown that \eqref{verthm-ad1} is non-negative for all $x\in E$.
\end{proof}

\begin{rem} \label{rem-inf-not-needed} In the case $r>0$ then the condition in \eqref{VerTHM1} is not necessary in order for the verification theorem to be true, since in this case \ref{regularity-verTHMbounded} is sufficient to obtain \eqref{verproof-okt} (using also our convention regarding expected values and infinite stopping times, as described in the beginning of Section \ref{sec:problem-formulation}).
\end{rem}

\begin{rem} The continuous differentiability requirement \ref{regularity-verTHM2} and requirement \ref{regularity-verTHMLipschitz} imply that we can approximate the function $f({\cdot},y)$ by the sequence of $\mathcal{C}^2$ functions $f_i({\cdot},y)$, on which we can apply the standard It\^o formula. After this we let $i\rightarrow \infty$ and effectively find that Dynkin's formula \eqref{Ver2} holds. We remark that the continuous differentiability requirement \ref{regularity-verTHM2} could in some settings be relaxed if we instead of using the current approach were to use a more general version of  It\^o's formula based on the concept of local time, see e.g. \cite{karatzas2012brownian,peskir2006optimal}.
\end{rem}

\begin{rem} Let us underline that we can now apply the standard procedure to use the verification theorem in order to find equilibrium value functions and equilibrium stopping times in particular cases. More precisely:
\begin{enumerate} [label=(\roman*)]
\item Make an ansatz, i.e. make an educated guess of how the solution $f(x,y)$ to the time-inconsistent variational inequalities should look like. The guess $f(x,y)$ should typically have traits in common with $F(x,y)$ and involve unspecified parameters, see e.g. the $a,b$ and $x^*$ in Example \ref{optimisticEx2} below.

\item Use the verification theorem to verify that $f(x,y)$ can solve the time-inconsistent variational inequalities (and the regularity conditions of the verification theorem) for some specific values of the parameter(s). Note that the point of step (ii) is two-fold 1) to make sure that the guess $f(x,y)$ has any chance of solving the time-inconsistent variational inequalities, and 2) to determine the unspecified parameters of $f(x,y)$. 

\item If the previous steps were successful then you may use the verification theorem to conclude that the guess $f(x,y)$, with the specified parameter(s), is indeed the auxiliary function, that $J(x):=f(x,x)$ is the corresponding equilibrium value function and that $\hat\tau:= \inf\{t\geq 0: X_t\notin C\}$ is the equilibrium stopping time, where $C:= \left\{x \in E: f(x,x) > F(x,x) \right\}$.
\end{enumerate}
\end{rem} 
 
In the setting of the present section we obtain the following result saying under sufficient regularity it holds for standard (time-consistent) stopping problems  that equilibrium stopping times are optimal.

\begin{thm} \label{newresB} Suppose $X$ is the strong solution to the SDE \eqref{SDE} and the function $F(x)$ in \eqref{newresA0} is non-negative and continous. Suppose an equilibrium stopping time $\hat\tau$ for the standard stopping problem \eqref{newresA0} exists and the corresponding auxiliary function (Definition \ref{def:equ_func}) is sufficiently regular to be a solution to the time-inconsistent variational inequalities and that conditions \eqref{regularity-verTHM} and \eqref{VerTHM1} are satisfied. Suppose the family 
$\{J_{\hat\tau }(X_\tau): \tau\leq \hat\tau\}$ is uniformly integrable, for each starting value $x \in E$, where $J_{\hat\tau }(x)$ is the equilibrium value function. Then, $\hat\tau$ is also an optimal stopping time for \eqref{newresA0}.
\end{thm}
\begin{proof} The reward function $F(x)$ in \eqref{newresA0} does not depend on $y$. Hence, the auxiliary function does not depend $y$, which means that it can be written as $f_{\hat\tau}(x)=\E_x(e^{-r{\hat\tau}}F(X_{\hat\tau}))=J_{\hat\tau}(x)$ --- in other words, in this case, the auxiliary function is equal to the equilibrium value function. It is now easy to see that if $J_{\hat\tau}(x)$ solves the time-inconsistent variational inequalities then it also solves the standard variational inequalities (or equivalently, free boundary problem) corresponding to the problem of optimal stopping in \eqref{newresA0}, cf. e.g. \cite[Theorem 10.4.1]{oksendal2013stochastic}; hence, standard verification arguments can be used to show that $J_{\hat\tau}(x)$ is in fact also the optimal value function (for the standard theory we refer to \cite{peskir2006optimal} and \cite[~ch. 10]{oksendal2013stochastic}). The result follows.
\end{proof}

\begin{ex} \label{optimisticEx2} Let us re-analyze the optimistic holder of the perpetual American option from Example \ref{optimistic-American} using the verification theorem. The advantage here is that we do not have to solve a sequence of free boundary problems, but can make a direct ansatz for the value function. Since the state process is a Wiener process it follows that $E=\R$. As $x\mapsto e^{-cx},\,x\mapsto e^{cx},\,c=\sqrt{2r},$ are the fundamental solutions to $A_Xf={rf}$ and due to symmetry, a natural guess for a solution to the time-inconsistent variational inequalities is
\[f(x,y)=\begin{cases}
x&,\;x\geq x^*,\\
ae^{-cx}+be^{cx}&,\;-x^*< x<x^*, y\geq0\\
0&,\;x\leq-x^*,
\end{cases}\]		
and $f(x,y)=f(-x,-y),\,y<0$, with the continuation region $C=(-x^*,x^*)$, for some parameters $a$, $b$ and $x^*$ to be determined. For $f{(\cdot,y)}$ to be continuous, we need
	\begin{align*}
	ae^{-cx^*}+be^{cx^*}&=x^*,\\
	ae^{cx^*}+be^{-cx^*}&=0.
	\end{align*}
For sufficient smoothness of $f(\cdot,x^*)$ we furthermore need 
	\[-cae^{-cx^*}+cbe^{cx^*}= 1.\]
	Elementary arguments yield that this system of equations indeed has a solution given by
\begin{align}
&f(x,y)=
\begin{cases}
x&,\;x\geq x^*,\\
\frac{1}{2}e^{cx^*}(x^*-\frac{1}{c})e^{-cx}+\frac{1}{2}e^{-cx^*}(x^*+\frac{1}{c})e^{cx}&,\;-x^*< x<x^*, y\geq0\\
0&,\;x\leq-x^*,
\end{cases}
\end{align}
where $x^*\in(0,1/c)$ satisfies
\begin{align}
-e^{4cx^*}(x^*-\frac{1}{c})=x^*+\frac{1}{c}. \label{guess-ex-} 
\end{align}
Using elementary methods one can now verify that:
\begin{itemize}
\item
$x \mapsto f(x,y)$ is convex and $\frac{\partial f(x^*,y)}{\partial x}=1$ for $x,y\geq 0$ which implies that 
$C= (-x^*,x^*)=\left\{x \in \R: f(x,x)-F(x,x)> 0 \right\}$, which also implies that \eqref{regularity-verTHM} and the condition in \eqref{VerTHM1} are fulfilled (we remark that the last condition is not necessary since $r>0$, cf. Remark \ref{rem-inf-not-needed}),
\item  conditions \eqref{freeB3-2}, \ref{regularity-verTHM-onC}, \ref{regularity-verTHMbounded} and \ref{regularity-verTHMLipschitz} hold. 
\end{itemize}
Let us explicitly verify condition \ref{regularity-verTHM2}. Naively taking derivatives gives us
\[\frac{\partial f(x,y)}{\partial x}=\begin{cases}
1&,\;x\geq x^*,\\
-\frac{c}{2}e^{cx^*}(x^*-\frac{1}{c})e^{-cx}+\frac{c}{2}e^{-cx^*}(x^*+\frac{1}{c})e^{cx}&,\;-x^*< x<x^*, y\geq0\\
0&,\;x\leq-x^*,
\end{cases}\]
\[\frac{\partial f(x,y)}{\partial x}=\begin{cases}
0&,\;x\geq x^*,\\
\frac{c}{2}e^{cx^*}(x^*-\frac{1}{c})e^{cx}-\frac{c}{2}e^{-cx^*}(x^*+\frac{1}{c})e^{-cx}&,\;-x^*<x<x^*, y<0\\
-1&,\;x\leq-x^*.
\end{cases}\]
It is easy to check that these derivatives are well defined except at points $(x,y)$ satisfying $(x,y)=(x^*,y)$ with  $y<0$ or $(x,y)=(-x^*,y)$ with $y\geq 0$. Hence, for fixed $y$, $f(\cdot,y)\in \mathcal{C}^1({B}(y,\epsilon))$, for a sufficiently small $\epsilon>0$. 
Naively taking derivatives again gives us
\[\frac{\partial^2 f(x,y)}{\partial x^2}=\begin{cases}
0&,\;x> x^*,\\
\frac{c^2}{2}e^{cx^*}(x^*-\frac{1}{c})e^{-cx}+\frac{c^2}{2}e^{-cx^*}(x^*+\frac{1}{c})e^{cx}&,\;-x^*< x<x^*, y\geq0\\
0&,\;x<-x^*,
\end{cases}\]
\[\frac{\partial^2 f(x,y)}{\partial x^2}=\begin{cases}
0&,\;x> x^*,\\
\frac{c^2}{2}e^{cx^*}(x^*-\frac{1}{c})e^{cx}+\frac{c^2}{2}e^{-cx^*}(x^*+\frac{1}{c})e^{-cx}&,\;-x^*<x<x^*, y<0\\
0&,\;x<-x^*.
\end{cases}\]
We thus see that $f(\cdot,y)\in \mathcal{C}^2({B}(y,\epsilon)\backslash \partial C)$, for any $y\in E$, and that this derivative is locally bounded. 
 We have thus verified \ref{regularity-verTHM2}. Now use that $c=\sqrt{2r}$ and the above to obtain\\
\noindent $A_X f(x,y) -rf(x,y)=\frac{1}{2}\frac{\partial^2f(x,y)}{\partial x^2}-rf(x,y)$
\[=\begin{cases}
-rx<0&,\;x>x^*,\\
0&,\;-x^*< x<x^*, y\geq0\\
0&,\;x<-x^*,
\end{cases}\]
\noindent $A_X f(x,y) -rf(x,y)=\frac{1}{2}\frac{\partial^2f(x,y)}{\partial x^2}-rf(x,y)$
\[=\begin{cases}
0&,\;x>x^*,\\
0&,\;-x^*< x<x^*, y<0\\
rx<0&,\;x<-x^*.
\end{cases}\]
This means that \eqref{freeB1}, \eqref{freeB3} and \ref{cond5} are also satisfied. We have thus verified that the function $f(x,y)$ 
{with $x^*$ determined in} \eqref{guess-ex-} is a solution to the time-inconsistent variational inequalities. The verification theorem therefore implies that $\hat\tau=\inf\{t\geq0: X_t \notin (-x^*,x^*)\}$ is an equilibrium stopping time {and that the corresponding equilibrium value function is}
\begin{align}
&J(x)=
\begin{cases}
x&,\;x\geq x^*,\\
\frac{1}{2}e^{cx^*}(x^*-\frac{1}{c})e^{-cx}+\frac{1}{2}e^{-cx^*}(x^*+\frac{1}{c})e^{cx}&,\;0\leq x<x^*,\\
\frac{1}{2}e^{cx^*}(x^*-\frac{1}{c})e^{cx}+\frac{1}{2}e^{-cx^*}(x^*+\frac{1}{c})e^{-cx}&,\;-x^*< x<0,\\
-x&,\;x\leq-x^*.
\end{cases}
\end{align}
\end{ex}

\begin{ex}\label{eq-selling-strat} \textbf{Equilibrium selling strategies under endogenous habit formation and exponential utility}.
We will now study a model for selling strategies under exponential utility and endogenous habit formation, using the verification theorem. Section \ref{sec:prev-lit} contains information about previous literature on related problems. 

Consider an investor who wishes to optimally dispose of an asset in a Black-Scholes market. Specifically, the price of the asset, measured in e.g USD or MUSD, is given by the process $X$ satisfying
\begin{align}
dX_t=\sigma X_tdW_t.
\end{align}
We model the utility of the investor as exponential, but we also let her utility be inversely related to the present price of the asset, which makes the problem time-inconsistent. Specifically, we assume that the agent wishes to maximize
\begin{align}
\E_x\left(e^{-r\tau}\left(1-e^{-a(X_\tau+g(x)-k)}\right)\right),
\end{align}
where $a,r,k>0$ are constants and $g:[0,\infty)\rightarrow \mathbb{R}$ is a non-increasing bounded function such that $x \mapsto x+g(x)$ is non-decreasing and $g(0)=0$.

We will study this endogenous habit formation selling problem without making any functional assumptions for $g(\cdot)$. In Figure \ref{fig-habit}, we present the solution to the problem for a particular specification of $g(\cdot)$.
\begin{rem} The reward function of the present model corresponds to the function $F(x,y):= 1-e^{-a(x+g(y)-k)}$, which is clearly time-inconsistent and bounded on $[0,\infty)^2$. If $g(\cdot)=0$ and $k=0$ then we recover a standard exponential utility function.
\end{rem}

\begin{rem} We interpret this model as the investor having formed a habit regarding what she thinks the asset should be worth, and the larger the current value of the asset is the less happy she will be for a given selling price in the future. The parameter $a$ is a measure of the risk aversion of the investor: a larger $a$ means more risk aversion.  The parameter $r$ is a measure of the impatience of the investor: a larger $r$ means more impatience. The nonstandard feature of this model is the function $g(\cdot)$ which we interpret to be a measure of the habit formation of the investor. The assumption that $g(\cdot)$ is non-increasing is interpreted as follows: the smaller the current price $x$ is, the happier the investor is given the same future selling price. The assumption that $x \mapsto x+g(x)$ is increasing means that the investor cannot become less happy for a larger selling price given immediately selling. The parameter $k$ allows the possibility for negative utility.
\end{rem}

A reasonable starting point is to try with a one-sided solution $C=(0,x^*)$. We therefore guess that the auxiliary function is
\[f(x,y)=\begin{cases}
1-e^{-a(x+g(y)-k)}&,\;x\geq x^*,\\
\E_x\left(e^{-r\tau_{[x^*,\infty)}}\left(1-e^{-a\left({X_{\tau_{[x^*,\infty)}}}+g(y)-k\right)}\right)\right)&,\;0<x< x^*,
\end{cases}\]
for some $x^*$ to be determined. Using standard theory, see e.g \cite[~ch. 9,10]{oksendal2013stochastic}, we note that the function $f(x,y)$ can be simplified using
\begin{align}
\E_x\left(e^{-r\tau_{[x^*,\infty)}}\left(1-e^{-a\left({X_{\tau_{[x^*,\infty)}}}+g(y)-k\right)}\right)\right)
= \left(\frac{x}{x^*}\right)^\gamma\left(1-e^{-a( {x^*}+g(y)-k)}\right)
\end{align}
where $\gamma=\frac{1}{2}+\sqrt{\frac{1}{4}+\frac{2r}{\sigma^2}}$. Naively taking derivatives therefore gives us
\[\frac {\partial f(x,y)}{\partial x}=\begin{cases}
ae^{-a(x+g(y)-k)}&,\;x\geq x^*,\\
\gamma\frac{x^{\gamma-1}}{{x^*}^\gamma}\left(1-e^{-a(x^*+g(y)-k)}\right)&,\;0<x< x^*.
\end{cases}\]
In order for \ref{regularity-verTHM2} to be fulfilled $x^*$ must satisfy
\begin{align} \label{exp-example4}
x^*ae^{-a(x^*+g(x^*)-k)} = 
\gamma\left(1-e^{-a(x^*+g(x^*)-k)}\right)
\end{align}
which means that $x^*$ must be the zero of the function
\begin{align}\label{exp-example3}
H(x)= \gamma- e^{-a(x +g(x)-k)}(\gamma+ax),
\end{align}
which must be verified to exist uniquely in $(0,\infty)$ for the particular choice of $g(\cdot)$. Note that a unique $x^*$ exists if there is no habit formation i.e. with $g(\cdot)=0$; to see this note that if $g(\cdot)=0$ then $H(0)=\gamma(1-e^{ak})<0$ and $H'(x)= ae^{-a(x-k)}(ax+\gamma-1)>0$ on $[0,\infty)$, since $\gamma>1$.

Taking derivatives again gives us
\[\frac {\partial^2 f(x,y)}{\partial x^2}=\begin{cases}
 -a^2e^{-a(x+g(y)-k)}&,\;x> x^*,\\
\gamma(\gamma-1)\frac{x^{\gamma-2}}{{x^*}^\gamma}\left(1-e^{-a(x^*+g(y)-k)}\right)&,\;0<x< x^*.
\end{cases}\]   
Using $A_X f(x,y) = \frac{1}{2}x^2\sigma^2\frac{\partial^2f(x,y)}{\partial x^2}$ and $\gamma(\gamma-1)= \frac{2r}{\sigma^2}$, we obtain
 
\noindent $A_X f(x,y) -r f(x,y)=$
\[=\begin{cases}
-\frac{x^2}{2}\sigma^2a^2e^{-a(x+g(y)-k)} - r\left(1-e^{-a(x+g(y)-k)}\right)&,\;x> x^*,\\
\frac{x^2}{2}\sigma^2\frac{2r}{\sigma^2}\frac{x^{\gamma-2}}{{x^*}^\gamma}\left(1-e^{-a(x^*+g(y)-k)}\right) 
-r\left(\frac{x}{x^*}\right)^\gamma\left(1-e^{-a( {x^*}+g(y)-k)}\right)&,\;0<x< x^*,
\end{cases}\]
\[=\begin{cases}
-\frac{x^2}{2}\sigma^2a^2e^{-a(x-g(y)-k)} - r\left(1-e^{-a(x+g(y)-k)}\right)&,\;x> x^*,\\
0&,\;0<x< x^*,
\end{cases}\]
which implies that $f(x,y)$ satisfies \eqref{freeB3}. It follows from \eqref{exp-example4} that $x^*+g(x^*)-k>0$, which since $x+g(x)$ is non-decreasing implies that
\[A_X f(x,x) -rf(x,x) =
-\frac{x^2}{2}\sigma^2a^2e^{-a(x+g(x)-k)} - r\left(1-e^{-a(x+g(x)-k)}\right)<0, \enskip x>x^*.
\]
Hence, \eqref{freeB1} is satisfied. Condition \ref{cond5} is verified in the same way. Conditions \ref{regularity-verTHM-onC}, \ref{regularity-verTHM2},  \ref{regularity-verTHMbounded} and \ref{regularity-verTHMLipschitz} are directly verified. Now, if $g(\cdot)$ is such that
\begin{align}\label{exp-example2}
C:=(0,x^*)= \left\{x \in \R: f(x,x)-F(x,x)> 0 \right\}
\end{align}
holds then conditions \eqref{freeB3-2} and \eqref{regularity-verTHM}  follow, and all the conditions of the verification theorem are hence fulfilled (the condition in \eqref{VerTHM1} is not necessary in this case, cf. Remark \ref{rem-inf-not-needed}).

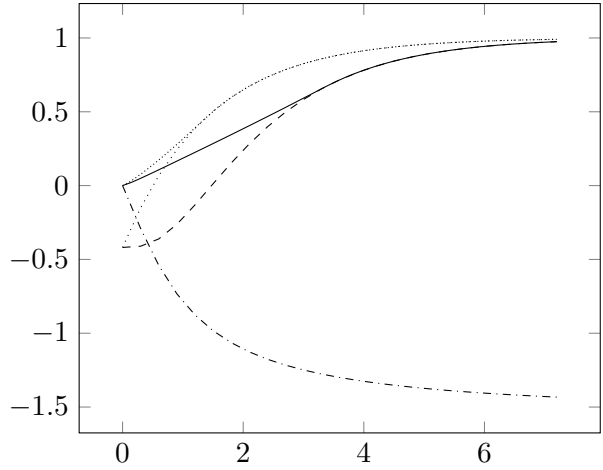
\begin{figure}[h]
	\begin{center}
	\begin{tikzpicture} 
	\begin{axis} 
	   	\addplot[style=dashdotted,domain=0:7.2] {rad(90-atan(x))-3.14159265359/2}; 
			\addplot[style=dashed,domain=0:7.2] {1-exp(-0.7*(x+rad(90-atan(x))-3.14159265359/2-0.5))}; 
			\addplot[style=solid,domain=0:3.3524]
			{(x/3.3523990)^(1.17082039)*(1-exp(-0.7*(3.3523990+rad(90-atan(x))-3.14159265359/2-0.5)))}; 
      \addplot[style=solid,domain=3.3524:7.2]{1-exp(-0.7*(x+rad(90-atan(x))-3.14159265359/2-0.5))}; 
			\addplot[style=dotted,domain=0:7.2] {1-exp(-0.7*(x-0.5))}; 
			\addplot[style=densely dotted,domain=0:1.341152]
			{(x/1.341152)^(1.17082039)*(1-exp(-0.7*(1.341152-0.5)))}; 
			\addplot[style=densely dotted,domain=1.341152:7.2] {1-exp(-0.7*(x-0.5))}; 
	\end{axis} 
	\end{tikzpicture} 
		\end{center}
			\caption{$x\mapsto J(x)$ (solid) and $x\mapsto F(x,x)$ (dashed), with $g(x)=\mbox{arccot}(x) - \frac{\pi}{2}$, here $x^*\approx 3.3524$. $x\mapsto J(x)$ (densely dotted) and $x\mapsto F(x,x)$ (dotted), with $g(x)=0$, here $x^*\approx 1.3412$. $x\mapsto \mbox{arccot}(x) - \frac{\pi}{2}$ (dash-dotted). $a=0.7,r=0.1,k=0.5$ and $\sigma=1$.}\label{fig-habit}
\end{figure}

To show that \eqref{exp-example2} holds for the case $g(\cdot)=0$ it is sufficient to show that 
\[\left(\frac{x}{x^*}\right)^\gamma\left(1-e^{-a(x^*-k)}\right)>1-e^{-a(x-k)} , \enskip  0<x<x^*.\]
This is trivially true if the right side is non-positive since the left side is positive by \eqref{exp-example4}, and we may thus treat the right side and the left side as positive. It is therefore sufficient to show that
\begin{align}
\kappa(x):= \frac{{x^*}^\gamma}{1-e^{-a(x^*-k)}}x^{-\gamma}\left(1-e^{-a(x-k)}\right)<1, \enskip 0<x<x^*.
\end{align}
We obtain that $\kappa'(x)= \frac{{x^*}^\gamma}{1-e^{-a(x^*-k)}}x^{-\gamma-1}(-H(x))$, and since $x^*$ is assumed to be the unique zero of $H(\cdot)$ and $H(0)<0$ it follows that $\kappa'(x)>0$, where we also used \eqref{exp-example4} to see that the first fraction in $\kappa(x)$ is positive. Since $\kappa(x^*)=1$ it follows that $\kappa(x^*)<1$ for $0<x<x^*$ and we are done. In order to show that \eqref{exp-example2} holds when $g(\cdot) \neq 0$, we must show that
\[\left(\frac{x}{x^*}\right)^\gamma\left(1-e^{-a(x^*-k)}e^{-ag(x)}\right)> 1-e^{-a(x-k)}e^{-ag(x)}, \enskip 0<x<x^*.\]
This is trivially true if the right side is non-positive; to see this use that the left side is positive by \eqref{exp-example4} and since $g(\cdot)$ is non-increasing. We may thus treat both the left and the right sides as positive; and since $1-e^{-a(x-k)}>1-e^{-a(x-k)}e^{-ag(x)}$ we may also treat $1-e^{-a(x-k)}$ as positive. It is thus enough to show that 
$\left(\frac{x^*}{x}\right)^\gamma \frac{1-e^{-a(x-k)}e^{-ag(x)}}{1-e^{-a(x^*-k)}e^{-ag(x)}}<1, \enskip 0<x<x^*$. But
$\frac{1-e^{-a(x-k)}e^{-ag(x)}}{1-e^{-a(x^*-k)}e^{-ag(x)}}<\frac{1-e^{-a(x-k)}}{1-e^{-a(x^*-k)}}$ and the result follows from the case when $g(\cdot)=0$.

We conclude that if $g(\cdot)$ is such that $H(\cdot)$ in \eqref{exp-example3} has a unique zero $x^*$, then the equilibrium stopping time is $\hat\tau=\inf\{t\geq0: X_t \geq x^*\}$ and the equilibrium value function is
\[J(x)=\begin{cases}
1-e^{-a(x+g(x)-k)}&,\;x\geq x^*,\\
\left(\frac{x}{x^*}\right)^\gamma\left(1-e^{-a( {x^*}+g(x)-k)}\right)&,\;0<x< x^*.
\end{cases}\]

\end{ex}

\bibliographystyle{abbrv}
\bibliography{time-incon_stopping}

\end{document}